\documentclass[11pt,letterpaper]{amsart}

\usepackage{amsmath}
\usepackage{amssymb}
\usepackage{ulem}
\usepackage{bm}
\usepackage{color}
\usepackage{graphicx}
\usepackage{mathrsfs}
\usepackage{enumerate}
\usepackage{footmisc}
\usepackage{blkarray}
\usepackage{framed}
\usepackage{caption}

\newtheorem{theorem}{Theorem}[section]

\newtheorem{corollary}[theorem]{Corollary}

\theoremstyle{definition}
\newtheorem{definition}[theorem]{Definition}
\newtheorem{example}[theorem]{Example}

\theoremstyle{remark}
\newtheorem{remark}[theorem]{Remark}

\numberwithin{equation}{section}

\begin{document}

\title{The grand picture behind Jensen's inequality}

\author{Jun Liu}
\address{School of Mathematics, Taiyuan University of Technology, Jinzhong 030600, Shanxi, China}
\email{liujun03@tyut.edu.cn}

\subjclass[2010]{Primary 26D15; Secondary 26A51}

\keywords{fundamental theorem of weighted means, fundamental theorem of systems of point masses, Jensen's inequality, Jensen gap, convex at a point, weighted convex at a point, Hermite-Hadamard integral inequality, squeeze theorem for convex functions, conditional convex hull, union of intervals is an interval test}

\begin{abstract}

Let $I$ and $J$ be two intervals, and let $f, g: I \rightarrow \mathbb{R}$.
If for any points $a$ and $b$ in $I$ and any positive numbers $p$ and $q$ such that $p + q = 1$, we have
\begin{align}
\nonumber
p f(a) + q f(b) + g(pa + qb) \in J,
\end{align}
then for any points $x_{1}, \ldots, x_{n}$ in $I$ and any positive numbers $\lambda_{1}, \ldots, \lambda_{n}$ such that $\sum_{i=1}^{n}\lambda_{i} = 1$, we have
\begin{align}
\nonumber
\sum_{i=1}^{n}\lambda_{i} f(x_{i}) + g\left( \sum_{i=1}^{n}\lambda_{i}x_{i} \right) \in J.
\end{align}
If we take $g = -f$ and $J = [0, +\infty)$, then the Jensen's inequality.

The conclusion is only a short glimpse of the grand picture behind Jensen's inequality shows in this paper.
\end{abstract}

\maketitle

\section{Introduction}

The history behind Jensen's inequality is quite a long and difficult story \cite{08,07}.
As early as 1875, J. Grolous \cite{06} had proved the following inequality:
\begin{align}
\nonumber
\frac{ \Phi(x_{1}) + \cdots + \Phi(x_{n}) }{n} \geq \Phi\left( \frac{ x_{1} + \cdots + x_{n} }{n} \right)
\end{align}
for all convex functions $\Phi: [a, b] \rightarrow \mathbb{R}$ and real numbers $x_{1}, \ldots, x_{n} \in [a, b]$.
In 1889 the following inequality was proven by O. H{\"o}lder \cite{09}:
\begin{align}
\nonumber
p_{1}\Phi(x_{1}) + \cdots + p_{n}\Phi(x_{n}) \geq \Phi( p_{1}x_{1} + \cdots + p_{n}x_{n} )
\end{align}
for all convex functions $\Phi: [a, b] \rightarrow \mathbb{R}$, and real numbers $x_{1}, \ldots, x_{n} \in [a, b]$ and positive numbers $p_{1}, \ldots, p_{n}$ such that $p_{1} + \cdots + p_{n} = 1$.
The 1906 paper of J.L.W.V. Jensen \cite{10} shows the following result on integrals, in addition to the result on sums:
\begin{align}
\nonumber
\frac{1}{b - a}\int_{a}^{b}\Phi( u(x) )\mathrm{d}x \geq \Phi\left( \frac{1}{b - a}\int_{a}^{b} u(x) \mathrm{d}x \right)
\end{align}
with $\Phi: I \rightarrow \mathbb{R}$ an integrable convex function and $u: [a, b] \rightarrow I$ an integrable function, where $I$ is an interval.

For almost a century and a half, Jensen's inequality and convex functions have been thought to be closely related.
However, this is just a coincidence.
There exists a more fundamental rule behind this, we call it the {\it fundamental theorem of weighted means}, which is an inherent property of real numbers.
(See Theorems \ref{0227002}, \ref{0406077} and \ref{0902001}.)


Suppose that $a_{1}, \ldots, a_{n} \in \mathbb{R}$ and $\lambda_{1}, \ldots, \lambda_{n} > 0, \lambda_{1} + \cdots + \lambda_{n} = 1$.

Let $f: \mathrm{dom}f \subseteq \mathbb{R} \rightarrow \mathbb{R}$ such that $a_{1}, \ldots, a_{n} \in \mathrm{dom}f$, and let $g: \mathrm{dom}g \subseteq \mathbb{R} \rightarrow \mathbb{R}$ such that $\sum_{k=1}^{n}\lambda_{k}a_{k} \in \mathrm{dom}g$. Then (See Theorem \ref{0227002})
\begin{align}
\nonumber
\left\{\begin{array}{c}
\sum_{k=1}^{n}\lambda_{k} f(a_{k}) + g( \sum_{k=1}^{n}\lambda_{k}a_{k} ) \\
= \sum_{s=1}^{m} w_{s}\Big[ p_{s} f(a_{ i_{s} }) + q_{s} f(a_{ j_{s} }) + g( p_{s}a_{ i_{s} } + q_{s}a_{ j_{s} } ) \Big], \\
p_{s}a_{ i_{s} } + q_{s}a_{ j_{s} } = \sum_{k=1}^{n}\lambda_{k}a_{k}, \text{where}\ p_{s}, q_{s} > 0, p_{s} + q_{s} = 1, \\
\{ i_{s} \mid s = 1, \ldots, m \} \cup \{ j_{s} \mid s = 1, \ldots, m \} = \{ 1, \ldots, n \}, \\
w_{1}, \ldots, w_{m} > 0, w_{1} + \cdots + w_{m} = 1.
\end{array}\right.
\end{align}
Therefore, the conclusion in the abstract is clear. (See Theorem \ref{0209009}.)

We also established the {\it fundamental theorem of systems of point masses}:
Every system of $n > 3$ point masses in three-dimensional space can be decomposed into at most $n-3$ systems of $4$ point masses, and all these systems of $4$ point masses have the same center of mass as the original one.
(See Theorems \ref{0219001}-\ref{0227002} and \ref{0907002}-\ref{0908009}.)

\section{The fundamental theorem of weighted means}

\begin{definition}\label{0710001}
Let $a_{1}, \ldots, a_{n}$ be real numbers and let $m_{1}, \ldots, m_{n}$ be positive numbers.
Define $\left[\begin{array}{ccc}
m_{1} & \cdots & m_{n} \\
a_{1} & \cdots & a_{n}
\end{array}\right]$ to be a system of $n$ point masses on the real line, where $m_{i}$ is the mass of point $i$, and $a_{i}$ is the displacement of point $i$ ($i = 1, \ldots, n$).
The center of mass (or {\it CM} for short) is given as:
\begin{align}
\nonumber
\left[\begin{array}{ccc}
m_{1} & \cdots & m_{n} \\
a_{1} & \cdots & a_{n}
\end{array}\right]_{CM} = \frac{ m_{1}a_{1} + \cdots + m_{n}a_{n} }{ m_{1} + \cdots + m_{n} }.
\end{align}
\end{definition}


\begin{theorem}\label{0219001}
Let $n \in \mathbb{N}$ and $n > 1$, and let $\lambda_{1}, \ldots, \lambda_{n} > 0$ such that
\begin{align}
\nonumber
a_{1} < \cdots < a_{z} < \frac{ \lambda_{1}a_{1} + \cdots + \lambda_{n}a_{n} }{ \lambda_{1} + \cdots + \lambda_{n} } < a_{z+1} < \cdots < a_{n}.
\end{align}
Then there are unique
\begin{align}
\begin{split}
\left\{\begin{array}{c}
\begin{minipage}[t]{10.6cm}
\begin{enumerate}[$\bm{\arabic{enumi}.}$]
\item $m \in \mathbb{N}$ such that $n/2 + \vert z - n/2 \vert \leq m \leq n-1$,
\item surjective functions
$$\begin{aligned}
i: \{ 1, \ldots, m \} \rightarrow &\{1, \ldots, z\}, \\
                        s \mapsto &i_{s}
\end{aligned}
\begin{gathered}
\ \ \text{and}\ \
\end{gathered}
\begin{aligned}
j: \{ 1, \ldots, m \} \rightarrow &\{z+1, \ldots, n\}, \\
                        s \mapsto &j_{s}
\end{aligned}$$ \\
where $i_{s}$ is decreasing\footnotemark[1] and $j_{s}$ is increasing such that $j_{s} - i_{s}$ is strictly increasing,
\item $p_{s}, q_{s} > 0, s = 1, \ldots, m$,
\end{enumerate}
\end{minipage}
\end{array}\right.
\end{split}\label{0228001}
\end{align}
such that
\begin{align}
\begin{split}
\left\{\begin{array}{cc}
\text{\begin{blockarray}{c}
\begin{block}{[c]}
$\lambda_{1}$ \\
$\vdots$ \\
$\lambda_{n}$ \\
\end{block}
\end{blockarray}} = \sum_{s=1}^{m}\text{\begin{blockarray}{cc}
\begin{block}{[c]c}
$0$      & \\
$\vdots$ & \\
$0$      & \\
$p_{s}$  & $\hspace{-4.5mm}\rightarrow$the $i_{s}$th row \\
$0$      & \\
$\vdots$ & \\
$0$      & \\
$q_{s}$  & $\hspace{-4.5mm}\rightarrow$the $j_{s}$th row \\
$0$      & \\
$\vdots$ & \\
$0$      & \\
\end{block}
\end{blockarray}}, \\
\frac{ p_{s}a_{ i_{s} } + q_{s}a_{ j_{s} } }{ p_{s} + q_{s} } = \frac{ \lambda_{1}a_{1} + \cdots + \lambda_{n}a_{n} }{ \lambda_{1} + \cdots + \lambda_{n} }, s = 1, \ldots, m.
\end{array}\right.
\end{split}\label{0226001}
\end{align}
\end{theorem}
\footnotetext[1]{In this paper increasing means nondecreasing; decreasing means nonincreasing.}
\begin{proof}Clearly, the theorem holds for $n=2$.
If the theorem holds for $2 \leq n \leq k$, then for $n=k+1$, we have
\begin{enumerate}[$\bm{\arabic{enumi}.}$]
\item $\frac{ \lambda_{1}a_{1} + \lambda_{k+1}a_{k+1} }{ \lambda_{1} + \lambda_{k+1} } = \frac{ \lambda_{1}a_{1} + \lambda_{2}a_{2} + \cdots + \lambda_{k}a_{k} + \lambda_{k+1}a_{k+1} }{ \lambda_{1} + \lambda_{2} + \cdots + \lambda_{k} + \lambda_{k+1} }$.

We have
\begin{align}
\nonumber
\left\{\begin{array}{c}
\text{\begin{blockarray}{c}
\begin{block}{[c]}
$\lambda_{1}$ \\
$\lambda_{2}$ \\
$\vdots$ \\
$\lambda_{k}$ \\
$\lambda_{k+1}$ \\
\end{block}
\end{blockarray}} = \text{\begin{blockarray}{c}
\begin{block}{[c]}
$0$ \\
$\lambda_{2}$ \\
$\vdots$ \\
$\lambda_{k}$ \\
$0$ \\
\end{block}
\end{blockarray}} + \text{\begin{blockarray}{c}
\begin{block}{[c]}
$\lambda_{1}$ \\
$0$ \\
$\vdots$ \\
$0$ \\
$\lambda_{k+1}$ \\
\end{block}
\end{blockarray}}, \\
\frac{ \lambda_{1}a_{1} + \lambda_{2}a_{2} + \cdots + \lambda_{k}a_{k} + \lambda_{k+1}a_{k+1} }{ \lambda_{1} + \lambda_{2} + \cdots + \lambda_{k} + \lambda_{k+1} } = \frac{ \lambda_{2}a_{2} + \cdots + \lambda_{k}a_{k} }{ \lambda_{2} + \cdots + \lambda_{k} } = \frac{ \lambda_{1}a_{1} + \lambda_{k+1}a_{k+1} }{ \lambda_{1} + \lambda_{k+1} }.
\end{array}\right.
\end{align}
\item $\frac{ \lambda_{1}a_{1} + \lambda_{k+1}a_{k+1} }{ \lambda_{1} + \lambda_{k+1} } < \frac{ \lambda_{1}^{\prime\prime} a_{1} + \lambda_{k+1} a_{k+1} }{ \lambda_{1}^{\prime\prime} + \lambda_{k+1} } := \frac{ \lambda_{1}a_{1} + \lambda_{2}a_{2} + \cdots + \lambda_{k}a_{k} + \lambda_{k+1}a_{k+1} }{ \lambda_{1} + \lambda_{2} + \cdots + \lambda_{k} + \lambda_{k+1} }$.

Clearly, $0 < \lambda_{1}^{\prime\prime} < \lambda_{1}$. Let $\lambda_{1}^{\prime} = \lambda_{1} - \lambda_{1}^{\prime\prime}$. Then
\begin{align}
\nonumber
\left\{\begin{array}{c}
\text{\begin{blockarray}{c}
\begin{block}{[c]}
$\lambda_{1}$ \\
$\lambda_{2}$ \\
$\vdots$ \\
$\lambda_{k}$ \\
$\lambda_{k+1}$ \\
\end{block}
\end{blockarray}} = \text{\begin{blockarray}{c}
\begin{block}{[c]}
$\lambda_{1}^{\prime}$ \\
$\lambda_{2}$ \\
$\vdots$ \\
$\lambda_{k}$ \\
$0$ \\
\end{block}
\end{blockarray}} + \text{\begin{blockarray}{c}
\begin{block}{[c]}
$\lambda_{1}^{\prime\prime}$ \\
$0$ \\
$\vdots$ \\
$0$ \\
$\lambda_{k+1}$ \\
\end{block}
\end{blockarray}}, \\
\frac{ \lambda_{1}a_{1} + \lambda_{2}a_{2} + \cdots + \lambda_{k}a_{k} + \lambda_{k+1}a_{k+1} }{ \lambda_{1} + \lambda_{2} + \cdots + \lambda_{k} + \lambda_{k+1} } = \frac{ \lambda_{1}^{\prime}a_{1} + \lambda_{2}a_{2} + \cdots + \lambda_{k}a_{k} }{ \lambda_{1}^{\prime} + \lambda_{2} + \cdots + \lambda_{k} } = \frac{ \lambda_{1}^{\prime\prime}a_{1} + \lambda_{k+1}a_{k+1} }{ \lambda_{1}^{\prime\prime} + \lambda_{k+1} }.
\end{array}\right.
\end{align}
\item $\frac{ \lambda_{1}a_{1} + \lambda_{k+1}a_{k+1} }{ \lambda_{1} + \lambda_{k+1} } > \frac{ \lambda_{1} a_{1} + \lambda_{k+1}^{\prime\prime} a_{k+1} }{ \lambda_{1} + \lambda_{k+1}^{\prime\prime} } := \frac{ \lambda_{1}a_{1} + \lambda_{2}a_{2} + \cdots + \lambda_{k}a_{k} + \lambda_{k+1}a_{k+1} }{ \lambda_{1} + \lambda_{2} + \cdots + \lambda_{k} + \lambda_{k+1} }$.

Clearly, $0 < \lambda_{k+1}^{\prime\prime} < \lambda_{k+1}$. Let $\lambda_{k+1}^{\prime} = \lambda_{k+1} - \lambda_{k+1}^{\prime\prime}$. Then
\begin{align}
\nonumber
\left\{\begin{array}{c}
\text{\begin{blockarray}{c}
\begin{block}{[c]}
$\lambda_{1}$ \\
$\lambda_{2}$ \\
$\vdots$ \\
$\lambda_{k}$ \\
$\lambda_{k+1}$ \\
\end{block}
\end{blockarray}} = \text{\begin{blockarray}{c}
\begin{block}{[c]}
$0$ \\
$\lambda_{2}$ \\
$\vdots$ \\
$\lambda_{k}$ \\
$\lambda_{k+1}^{\prime}$ \\
\end{block}
\end{blockarray}} + \text{\begin{blockarray}{c}
\begin{block}{[c]}
$\lambda_{1}$ \\
$0$ \\
$\vdots$ \\
$0$ \\
$\lambda_{k+1}^{\prime\prime}$ \\
\end{block}
\end{blockarray}}, \\
\frac{ \lambda_{1}a_{1} + \lambda_{2}a_{2} + \cdots + \lambda_{k}a_{k} + \lambda_{k+1}a_{k+1} }{ \lambda_{1} + \lambda_{2} + \cdots + \lambda_{k} + \lambda_{k+1} } = \frac{ \lambda_{2}a_{2} + \cdots + \lambda_{k}a_{k} + \lambda_{k+1}^{\prime}a_{k+1} }{ \lambda_{2} + \cdots + \lambda_{k} + \lambda_{k+1}^{\prime} } = \frac{ \lambda_{1}a_{1} + \lambda_{k+1}^{\prime\prime}a_{k+1} }{ \lambda_{1} + \lambda_{k+1}^{\prime\prime} }.
\end{array}\right.
\end{align}
\end{enumerate}
By the way, $\max\{z, n-z\} = n/2 + \vert z-n/2 \vert$.
Thus, the assertion of the theorem has been verified by mathematical induction.
\end{proof}
\begin{flushleft}
$1^{0}$. A physical interpretation. \\
\begin{blockarray}{ccc}
\begin{block}{[ccc]}
$\lambda_{1}$  &  $\ldots$  &  $\lambda_{n}$ \\
      $a_{1}$  &  $\ldots$  &        $a_{n}$ \\
\end{block}
\end{blockarray}
can be decomposed into
\begin{blockarray}{cc}
\begin{block}{[cc]}
    $p_{s}$    &      $q_{s}$  \\
$a_{ i_{s} }$  &  $a_{ j_{s} }$ \\
\end{block}
\end{blockarray}, $s = 1, \ldots, m$, such that
\begin{blockarray}{cc}
\begin{block}{[cc]}
    $p_{s}$    &      $q_{s}$  \\
$a_{ i_{s} }$  &  $a_{ j_{s} }$ \\
\end{block}
\end{blockarray}$_{CM}$ = \begin{blockarray}{ccc}
\begin{block}{[ccc]}
$\lambda_{1}$  &  $\ldots$  &  $\lambda_{n}$ \\
      $a_{1}$  &  $\ldots$  &        $a_{n}$ \\
\end{block}
\end{blockarray}$_{CM}$, $s = 1, \ldots, m$. \\
$2^{0}$. $p_{s} \leq \lambda_{ i_{s} }$ and $q_{s} \leq \lambda_{ j_{s} }$, $s = 1, \ldots, m$. \\
$3^{0}$. $\sum_{k=1}^{n}\lambda_{k} = \sum_{s=1}^{m}( p_{s} + q_{s} )$. \\
$4^{0}$. $\frac{ \lambda_{1}a_{1} + \cdots + \lambda_{n}a_{n} }{ \lambda_{1} + \cdots + \lambda_{n} } \in ( a_{z}, a_{z+1} ) = ( a_{ i_{1} }, a_{ j_{1} } ) \subset \cdots \subset ( a_{ i_{m} }, a_{ j_{m} } ) = ( a_{1}, a_{n} )$. \\
$5^{0}$. The following statements are equivalent to (\ref{0226001}).
\begin{enumerate}[$\bm{\arabic{enumi}.}$]
\item $\left\{\begin{array}{c}
\lambda_{k} = \sum_{s=1}^{m}( p_{s} 1_{ i_{s} = k } + q_{s} 1_{ j_{s} = k } ), k = 1, \ldots, n, \\
\frac{ p_{s}a_{ i_{s} } + q_{s}a_{ j_{s} } }{ p_{s} + q_{s} } = \frac{ \lambda_{1}a_{1} + \cdots + \lambda_{n}a_{n} }{ \lambda_{1} + \cdots + \lambda_{n} }, s = 1, \ldots, m.
\end{array}\right.$
\item Let $f: \mathrm{dom}f \subseteq \mathbb{R} \rightarrow \mathbb{R}$ such that $a_{1}, \ldots, a_{n} \in \mathrm{dom}f$. Then
\begin{align}
\nonumber
\left\{\begin{array}{c}
\sum_{k=1}^{n}\lambda_{k} f(a_{k}) = \sum_{s=1}^{m}\Big( p_{s} f(a_{ i_{s} }) + q_{s} f(a_{ j_{s} }) \Big), \\
\frac{ p_{s}a_{ i_{s} } + q_{s}a_{ j_{s} } }{ p_{s} + q_{s} } = \frac{ \lambda_{1}a_{1} + \cdots + \lambda_{n}a_{n} }{ \lambda_{1} + \cdots + \lambda_{n} }, s = 1, \ldots, m.
\end{array}\right.
\end{align}
Let $w_{s} = \frac{ p_{s} + q_{s} }{ ( p_{1} + q_{1} ) + \cdots + ( p_{m} + q_{m} ) }, s = 1, \ldots, m$. Then
\begin{align}
\nonumber
\left\{\begin{array}{c}
\frac{ \lambda_{1} f(a_{1}) + \cdots + \lambda_{n} f(a_{n}) }{ \lambda_{1} + \cdots + \lambda_{n} } = \sum_{s=1}^{m} w_{s} \cdot \frac{ p_{s} f(a_{ i_{s} }) + q_{s} f(a_{ j_{s} }) }{ p_{s} + q_{s} }, \\
\frac{ p_{s}a_{ i_{s} } + q_{s}a_{ j_{s} } }{ p_{s} + q_{s} } = \frac{ \lambda_{1}a_{1} + \cdots + \lambda_{n}a_{n} }{ \lambda_{1} + \cdots + \lambda_{n} }, s = 1, \ldots, m, \\
w_{1}, \ldots, w_{m} > 0, w_{1} + \cdots + w_{m} = 1.
\end{array}\right.
\end{align}
Let $g: \mathrm{dom}g \subseteq \mathbb{R} \rightarrow \mathbb{R}$ such that $\frac{ \lambda_{1}a_{1} + \cdots + \lambda_{n}a_{n} }{ \lambda_{1} + \cdots + \lambda_{n} } \in \mathrm{dom}g$. Then
\begin{align}
\nonumber
\left\{\begin{array}{c}
\frac{ \lambda_{1} f(a_{1}) + \cdots + \lambda_{n} f(a_{n}) }{ \lambda_{1} + \cdots + \lambda_{n} } + g\left( \frac{ \lambda_{1}a_{1} + \cdots + \lambda_{n}a_{n} }{ \lambda_{1} + \cdots + \lambda_{n} } \right) \\
= \sum_{s=1}^{m} w_{s}\left[ \frac{ p_{s} f(a_{ i_{s} }) + q_{s} f(a_{ j_{s} }) }{ p_{s} + q_{s} } + g\left( \frac{ p_{s}a_{ i_{s} } + q_{s}a_{ j_{s} } }{ p_{s} + q_{s} } \right) \right], \\
\frac{ p_{s}a_{ i_{s} } + q_{s}a_{ j_{s} } }{ p_{s} + q_{s} } = \frac{ \lambda_{1}a_{1} + \cdots + \lambda_{n}a_{n} }{ \lambda_{1} + \cdots + \lambda_{n} }, s = 1, \ldots, m, \\
w_{1}, \ldots, w_{m} > 0, w_{1} + \cdots + w_{m} = 1.
\end{array}\right.
\end{align}
\end{enumerate}
\end{flushleft}


\begin{theorem}[The fundamental theorem of systems of point masses]\label{0710002}\hfill \\
Every system of $n > 1$ point masses on the real line can be decomposed into at most $n-1$ systems of two point masses, and all these systems of two point masses have the same center of mass as the original one.
\end{theorem}
\begin{flushleft}
$1^{0}$. We also call the theorem the {\it system of point masses decomposition theorem}.
\end{flushleft}


\begin{theorem}[The fundamental theorem of weighted means, part 1]\label{0227002}\hfill \\
Let $n \in \mathbb{N}$ and $n > 1$, and let $\lambda_{1}, \ldots, \lambda_{n} > 0, \lambda_{1} + \cdots + \lambda_{n} = 1$ such that
\begin{align}
\nonumber
a_{1} < \cdots < a_{z} < \lambda_{1}a_{1} + \cdots + \lambda_{n}a_{n} < a_{z+1} < \cdots < a_{n}.
\end{align}
Then there are unique
\begin{align}
\begin{split}
\left\{\begin{array}{c}
\begin{minipage}[t]{10.6cm}
\begin{enumerate}[$\bm{\arabic{enumi}.}$]
\item $m \in \mathbb{N}$ such that $n/2 + \vert z - n/2 \vert \leq m \leq n-1$,
\item $w_{1}, \ldots, w_{m} > 0, w_{1} + \cdots + w_{m} = 1$,
\item surjective functions
$$\begin{aligned}
i: \{ 1, \ldots, m \} \rightarrow &\{1, \ldots, z\}, \\
                        s \mapsto &i_{s}
\end{aligned}
\begin{gathered}
\ \ \text{and}\ \
\end{gathered}
\begin{aligned}
j: \{ 1, \ldots, m \} \rightarrow &\{z+1, \ldots, n\}, \\
                        s \mapsto &j_{s}
\end{aligned}$$ \\
where $i_{s}$ is decreasing and $j_{s}$ is increasing such that $j_{s} - i_{s}$ is strictly increasing,
\item $p_{s}, q_{s} > 0, p_{s} + q_{s} = 1, s = 1, \ldots, m$,
\end{enumerate}
\end{minipage}
\end{array}\right.
\end{split}\label{0228002}
\end{align}
such that
\begin{align}
\begin{split}
\left\{\begin{array}{cc}
\text{\begin{blockarray}{c}
\begin{block}{[c]}
$\lambda_{1}$ \\
$\vdots$ \\
$\lambda_{n}$ \\
\end{block}
\end{blockarray}} = \sum_{s=1}^{m} w_{s}\text{\begin{blockarray}{cc}
\begin{block}{[c]c}
$0$      & \\
$\vdots$ \\
$0$      & \\
$p_{s}$  &  $\hspace{-4.5mm}\rightarrow$the $i_{s}$th row \\
$0$      & \\
$\vdots$ & \\
$0$      & \\
$q_{s}$  &  $\hspace{-4.5mm}\rightarrow$the $j_{s}$th row \\
$0$      & \\
$\vdots$ & \\
$0$      & \\
\end{block}
\end{blockarray}}, \\
p_{s}a_{ i_{s} } + q_{s}a_{ j_{s} } = \lambda_{1}a_{1} + \cdots + \lambda_{n}a_{n}, s = 1, \ldots, m.
\end{array}\right.
\end{split}\label{0227003}
\end{align}
\end{theorem}
\begin{flushleft}
$1^{0}$. We also call the theorem the {\it first decomposition principle for weighted means}. \\
$2^{0}$. The following statements are equivalent to (\ref{0227003}).
\begin{enumerate}[$\bm{\arabic{enumi}.}$]
\item $\left\{\begin{array}{c}
\lambda_{k} = \sum_{s=1}^{m} w_{s}( p_{s} 1_{ i_{s} = k } + q_{s} 1_{ j_{s} = k } ), k = 1, \ldots, n, \\
p_{s}a_{ i_{s} } + q_{s}a_{ j_{s} } = \sum_{k=1}^{n}\lambda_{k}a_{k}, s = 1, \ldots, m.
\end{array}\right.$
\item Let $f: \mathrm{dom}f \subseteq \mathbb{R} \rightarrow \mathbb{R}$ such that $a_{1}, \ldots, a_{n} \in \mathrm{dom}f$. Then
\begin{align}
\left\{\begin{array}{c}
\sum_{k=1}^{n}\lambda_{k} f(a_{k}) = \sum_{s=1}^{m} w_{s}\Big( p_{s} f(a_{ i_{s} }) + q_{s} f(a_{ j_{s} }) \Big), \\
p_{s}a_{ i_{s} } + q_{s} a_{ j_{s} } = \sum_{k=1}^{n}\lambda_{k}a_{k}, s = 1, \ldots, m.
\end{array}\right.\label{1123001}
\end{align}
Let $g: \mathrm{dom}g \subseteq \mathbb{R} \rightarrow \mathbb{R}$ such that $\sum_{k=1}^{n}\lambda_{k}a_{k} \in \mathrm{dom}g$. Then
\begin{align}
\left\{\begin{array}{c}
\sum_{k=1}^{n}\lambda_{k} f(a_{k}) + g( \sum_{k=1}^{n}\lambda_{k}a_{k} ) \\
= \sum_{s=1}^{m} w_{s}\Big[ p_{s} f(a_{ i_{s} }) + q_{s} f(a_{ j_{s} }) + g( p_{s}a_{ i_{s} } + q_{s}a_{ j_{s} } ) \Big], \\
p_{s}a_{ i_{s} } + q_{s}a_{ j_{s} } = \sum_{k=1}^{n}\lambda_{k}a_{k}, s = 1, \ldots, m.
\end{array}\right.\label{0103001}
\end{align}
\end{enumerate}
\end{flushleft}


\begin{example}\label{0406042}
Let $f: \mathrm{dom}f \subseteq \mathbb{R} \rightarrow \mathbb{R}$ and $g: \mathrm{dom}g \subseteq \mathbb{R} \rightarrow \mathbb{R}$. \\
Suppose that $1, 2, 6 \in \mathrm{dom}f$ and $3 \in \mathrm{dom}g$. Then
\begin{align}
\nonumber
  &\frac{ f(1) + f(2) + f(6) }{3} = \frac{4}{9}\left( \frac{3}{4} f(2) + \frac{1}{4} f(6) \right) + \frac{5}{9}\left( \frac{3}{5} f(1) + \frac{2}{5} f(6) \right),  \\
\nonumber
  &\frac{ f(1) + f(2) + f(6) }{3} + g\left( \frac{1 + 2 + 6}{3} \right)  \\
\nonumber
= &\frac{4}{9}\left[ \frac{3}{4} f(2) + \frac{1}{4} f(6) + g\left( \frac{3}{4} \times 2 + \frac{1}{4} \times 6 \right) \right] + \frac{5}{9}\left[ \frac{3}{5} f(1) + \frac{2}{5} f(6) + g\left( \frac{3}{5} \times 1 + \frac{2}{5} \times 6 \right) \right].
\end{align}
\end{example}


\begin{theorem}[The fundamental theorem of weighted means, part 2]\label{0406077}\hfill \\
Let $c_{i} \in (a, b), w_{i} > 0, \sum_{i=1}^{n} w_{i} = 1$ and $p, q > 0, p + q = 1$ such that $\sum_{i=1}^{n} w_{i}c_{i} = pa + qb$.
Then the system of equations
\begin{align}
\nonumber
\left\{\begin{array}{c}
\sum_{i=1}^{n} p_{i} = p, p_{i} > 0, \sum_{i=1}^{n} q_{i} = q, q_{i} > 0, \\
p_{i} + q_{i} = w_{i}, \frac{ p_{i}a + q_{i}b }{ p_{i} + q_{i} } = c_{i}, i = 1, \ldots, n
\end{array}\right.
\end{align}
has a unique solution
\begin{align}
\nonumber
p_{i} = w_{i}\theta_{i}, q_{i} = w_{i}( 1 - \theta_{i} ), i = 1, \ldots, n,
\end{align}
where $\theta_{i}a + ( 1 - \theta_{i} )b = c_{i}$.
\end{theorem}
\begin{flushleft}
$1^{0}$. We also call the theorem the {\it second decomposition principle for weighted means}. \\
$2^{0}$. Let $f: \mathrm{dom}f \subseteq \mathbb{R} \rightarrow \mathbb{R}$ such that $a, b \in \mathrm{dom}f$. Then
\begin{align}
\nonumber
\left\{\begin{array}{c}
p f(a) + q f(b) = \sum_{i=1}^{n} w_{i} \cdot \frac{ p_{i} f(a) + q_{i} f(b) }{ p_{i} + q_{i} }, \\
\frac{ p_{i}a + q_{i}b }{ p_{i} + q_{i} } = c_{i}, i = 1, \ldots, n.
\end{array}\right.
\end{align}
Let $g: \mathrm{dom}g \subseteq \mathbb{R} \rightarrow \mathbb{R}$ such that $c_{1}, \ldots, c_{n} \in \mathrm{dom}g$. Then
\begin{align}
\nonumber
\left\{\begin{array}{c}
p f(a) + q f(b) + \sum_{i=1}^{n} w_{i} g(c_{i}) = \sum_{i=1}^{n} w_{i}\left[ \frac{ p_{i} f(a) + q_{i} f(b) }{ p_{i} + q_{i} } + g\left( \frac{ p_{i}a + q_{i}b }{ p_{i} + q_{i} } \right) \right], \\
\frac{ p_{i}a + q_{i}b }{ p_{i} + q_{i} } = c_{i}, i = 1, \ldots, n.
\end{array}\right.
\end{align}
\end{flushleft}


\begin{theorem}[The fundamental theorem of weighted means, part 3]\label{0902001}\hfill \\
Let $a_{i} \in \mathbb{R}, p_{i} > 0, \sum_{i=1}^{m} p_{i} = 1$ and $b_{j} \in \mathbb{R}, q_{j} > 0, \sum_{j=1}^{n} q_{j} = 1$
such that $\sum_{i=1}^{m} p_{i}a_{i} = \sum_{j=1}^{n} q_{j}b_{j}$. Suppose that $a_{i} \not\in ( \min_{j} b_{j}, \max_{j} b_{j} ), i = 1, \ldots, m$.
Then there are $\lambda_{ij} \geq 0$ satisfy $\lambda_{i1} + \cdots + \lambda_{in} = p_{i}$ and $\lambda_{1j} + \cdots + \lambda_{mj} = q_{j}$ such that
\begin{align}
\nonumber
\frac{ \lambda_{1j}a_{1} + \cdots + \lambda_{mj}a_{m} }{ \lambda_{1j} + \cdots + \lambda_{mj} } = b_{j}, j = 1, \ldots, n.
\end{align}
\end{theorem}
\begin{proof}By Theorems \ref{0227002} and \ref{0406077}, we conclude the proof.
\end{proof}
\begin{flushleft}
\begin{flushleft}
$1^{0}$. We also call the theorem the {\it third decomposition principle for weighted means}. \\
$2^{0}$. Let $f: \mathrm{dom}f \subseteq \mathbb{R} \rightarrow \mathbb{R}$ such that $a_{1}, \ldots, a_{m} \in \mathrm{dom}f$. Then
\begin{align}
\nonumber
\left\{\begin{array}{c}
\sum_{i=1}^{m} p_{i} f(a_{i}) = \sum_{j=1}^{n} q_{j} \cdot \frac{ \lambda_{1j} f(a_{1}) + \cdots + \lambda_{mj} f(a_{m}) }{ \lambda_{1j} + \cdots + \lambda_{mj} }, \\
\frac{ \lambda_{1j}a_{1} + \cdots + \lambda_{mj}a_{m} }{ \lambda_{1j} + \cdots + \lambda_{mj} } = b_{j}, j = 1, \ldots, n.
\end{array}\right.
\end{align}
Let $g: \mathrm{dom}g \subseteq \mathbb{R} \rightarrow \mathbb{R}$ such that $b_{1}, \ldots, b_{n} \in \mathrm{dom}g$. Then
\begin{align}
\nonumber
\left\{\begin{array}{c}
\sum_{i=1}^{m} p_{i} f(a_{i}) + \sum_{j=1}^{n} q_{j}g(b_{j}) \\
= \sum_{j=1}^{n} q_{j}\left[ \frac{ \lambda_{1j} f(a_{1}) + \cdots + \lambda_{mj} f(a_{m}) }{ \lambda_{1j} + \cdots + \lambda_{mj} } + g\left( \frac{ \lambda_{1j}a_{1} + \cdots + \lambda_{mj}a_{m} }{ \lambda_{1j} + \cdots + \lambda_{mj} } \right) \right], \\
\frac{ \lambda_{1j}a_{1} + \cdots + \lambda_{mj}a_{m} }{ \lambda_{1j} + \cdots + \lambda_{mj} } = b_{j}, j = 1, \ldots, n.
\end{array}\right.
\end{align}
\end{flushleft}
\end{flushleft}

\section{Convex at a point}

\begin{definition}\label{0330014}
Let $D \neq \emptyset$ be a subset of $\mathbb{R}$, and let $f: D \rightarrow \mathbb{R}$.
Suppose that $c \in D$ and $\emptyset \neq B \subseteq D$.
\begin{enumerate}[$\bm{\arabic{enumi}.}$]
\item If for any $a, b \in D$ and any $p, q > 0, p + q = 1$ such that $pa + qb = c$, we have
\begin{align}
\nonumber
p f(a) + q f(b) \geq f(pa + qb),
\end{align}
we say that $f$ is convex at $c$ over $D$.
If the equality holds if and only if $a = b$, we say that $f$ is strictly convex at $c$ over $D$.

\item If $f$ is (strictly) convex at every $x \in B$ over $D$, we say that $f$ is (strictly) convex on $B$ over $D$.
\item If $f$ is (strictly) convex on $D$ over $D$, we say that $f$ is (strictly) convex on $D$.
\end{enumerate}
\end{definition}


\begin{theorem}\label{0331009}
Let $D \neq \emptyset$ be a subset of $\mathbb{R}$, and let $f: D \rightarrow \mathbb{R}$. Then \\
$\bm{1.}$ $f$ is convex on $D$ if and only if for any $x_{1}, x_{2}, x_{3} \in D, x_{1} < x_{2} < x_{3}$, we have
\begin{align}
\nonumber
\frac{ f(x_{2}) - f(x_{1}) }{ x_{2} - x_{1} } \leq \frac{ f(x_{3}) - f(x_{1}) }{ x_{3} - x_{1} } \leq \frac{ f(x_{3}) - f(x_{2}) }{ x_{3} - x_{2} }.
\end{align}
$\bm{2.}$ $f$ is strictly convex on $D$ if and only if the inequalities above are strict.
\end{theorem}
\begin{flushleft}
$1^{0}$. All of these inequalities $\frac{ f(x_{2}) - f(x_{1}) }{ x_{2} - x_{1} } \leq \frac{ f(x_{3}) - f(x_{1}) }{ x_{3} - x_{1} }$, $\frac{ f(x_{2}) - f(x_{1}) }{ x_{2} - x_{1} } \leq \frac{ f(x_{3}) - f(x_{2}) }{ x_{3} - x_{2} }$, and $\frac{ f(x_{3}) - f(x_{1}) }{ x_{3} - x_{1} } \leq \frac{ f(x_{3}) - f(x_{2}) }{ x_{3} - x_{2} }$ are equivalent.
\end{flushleft}


\begin{theorem}\label{0605001}
Let $D \neq \emptyset$ be a subset of $\mathbb{R}$, and let $f$ be convex at $c \in D$ over $D$.
Then for any $x_{i} \in D$ and any $\lambda_{i} > 0, \sum_{i=1}^{n}\lambda_{i} = 1$ such that $\sum_{i=1}^{n}\lambda_{i}x_{i} = c$, we have
\begin{align}
\nonumber
\sum_{i=1}^{n}\lambda_{i} f(x_{i}) \geq f\left( \sum_{i=1}^{n}\lambda_{i}x_{i} \right).
\end{align}
If $f$ is strictly convex at $c$ over $D$, then the equality holds if and only if $x_{1} = \cdots = x_{n}$.
\end{theorem}
\begin{proof}By Theorem \ref{0227002}, we conclude the proof.
\end{proof}


\begin{theorem}\label{0401009}
Let $f \in \mathcal{R}[a, b]$ be convex at $c \in (a, b)$ over $[a, b]$.
Then for any $\lambda(x) \geq 0, \int_{a}^{b} \lambda(x) \mathrm{d}x = 1$ such that $\int_{a}^{b} \lambda(x) x \mathrm{d}x = c$, we have
\begin{align}
\nonumber
\int_{a}^{b} \lambda(x) f(x) \mathrm{d}x \geq f\left( \int_{a}^{b} \lambda(x) x \mathrm{d}x \right).
\end{align}
If $f$ is strictly convex at $c$ over $[a, b]$, then the inequality is strict.
\end{theorem}


\begin{theorem}\label{0608009}
Let $D \neq \emptyset$ be a subset of $\mathbb{R}$, and let $f$ be convex at $c \in D$ over $D$.
Then for any random variable $\nu \in D$ such that $\operatorname{\bf E}\nu = c$, we have
\begin{align}
\nonumber
\operatorname{\bf E} f(\nu) \geq f( \operatorname{\bf E}\nu ).
\end{align}
If $f$ is strictly convex at $c$ over $D$, then the equality holds if and only if $\operatorname{P}( \operatorname{\bf E}\nu ) = 1$.
\end{theorem}
\begin{flushleft}
$1^{0}$. In this paper, $\nu \in D$ means random variable $\nu$ taking values in $D$.
\end{flushleft}

\section{Weighted convex at a point}

\begin{definition}\label{0426001}
Let $D \neq \emptyset$ be a subset of $\mathbb{R}$, and let $f: D \rightarrow \mathbb{R}$.
Suppose that $c \in D^{*}$ and $\emptyset \neq B \subseteq D^{*}$.
\begin{enumerate}[$\bm{\arabic{enumi}.}$]
\item If for any $a_{i}, b_{i} \in D$ and any $p_{i}, q_{i} > 0, p_{i} + q_{i} = 1$ such that
\begin{align}
\nonumber
p_{0}a_{0} + q_{0}b_{0} = p_{1}a_{1} + q_{1}b_{1} = c, \text{where}\ [a_{0}, b_{0}] \subseteq [a_{1}, b_{1}],
\end{align}
we have
\begin{align}
\nonumber
p_{0} f(a_{0}) + q_{0} f(b_{0}) \leq p_{1} f(a_{1}) + q_{1} f(b_{1}),
\end{align}
we say that $f$ is weighted convex at $c$ over $D$. If the equality holds if and only if $[a_{0}, b_{0}] = [a_{1}, b_{1}]$, we say that $f$ is strictly weighted convex at $c$ over $D$.
\item If $f$ is (strictly) weighted convex at every $x \in B$ over $D$, we say that $f$ is (strictly) weighted convex on $B$ over $D$.
\item If $f$ is (strictly) weighted convex on $D^{*}$ over $D$, we say that $f$ is (strictly) weighted convex on $D$.
\end{enumerate}
\end{definition}
\begin{flushleft}
$1^{0}$. In this paper, $D^{*}$ is the convex hull of $D$. For more details, see \cite{03}.
\end{flushleft}


\begin{theorem}\label{0718001}
Let $D \neq \emptyset$ be a subset of $\mathbb{R}$, and let $f: D \rightarrow \mathbb{R}$. Suppose that $c \in D$.
If $f$ is (strictly) weighted convex at $c$ over $D$, then $f$ is (strictly) convex at $c$ over $D$.
\end{theorem}


\begin{example}\label{0427002}
Let $\operatorname{sinc}x$ equals $1$ when $x=0$ and $\sin x / x$ otherwise, then $y = -\operatorname{sinc}x$ is strictly convex at $x = 0$ over $\mathbb{R}$, but not weighted convex at $x = 0$ over $\mathbb{R}$.
\end{example}


\begin{theorem}\label{0716001}
Let $D \neq \emptyset$ be a subset of $\mathbb{R}$, and let $f: D \rightarrow \mathbb{R}$.
Then the following statements are equivalent.
\begin{enumerate}[$\bm{\arabic{enumi}.}$]
\item $f$ is (strictly) convex on $D$.
\item $f$ is (strictly) weighted convex on $D$.
\end{enumerate}
\end{theorem}
\begin{proof}By Theorem \ref{0406077}, we conclude the proof.
\end{proof}


\begin{theorem}\label{0605002}
Let $D \neq \emptyset$ be a subset of $\mathbb{R}$, and let $f$ be weighted convex at $c \in D^{*}$ over $D$.
Then for any
\begin{enumerate}[(1)]
\item $a_{i} \in D$ and $p_{i} > 0, \sum_{i=1}^{m} p_{i} = 1$,
\item $b_{j} \in D$ and $q_{j} > 0, \sum_{j=1}^{n} q_{j} = 1$,
\end{enumerate}
such that $\sum_{i=1}^{m} p_{i}a_{i} = \sum_{j=1}^{n} q_{j}b_{j} = c$, where $a_{i} \not\in ( \min_{j} b_{j}, \max_{j} b_{j} )$, we have
\begin{align}
\nonumber
\sum_{i=1}^{m} p_{i} f(a_{i}) \geq \sum_{j=1}^{n} q_{j} f(b_{j}).
\end{align}
If $f$ is strictly weighted convex at $c$ over $D$, then the equality holds if and only if $b_{j} = \min_{i} a_{i}\ \text{or}\ \max_{i} a_{i}, j = 1, \ldots, n$.
\end{theorem}
\begin{proof}By Theorem \ref{0227002}, we conclude the proof.
\end{proof}


\begin{theorem}\label{0711003}
Suppose $[ a_{0}, b_{0} ] \subset ( a_{1}, b_{1} )$, and let $f \in \mathcal{R}[ a_{1}, b_{1} ]$ be weighted convex at $c \in ( a_{0}, b_{0} )$ over $[ a_{1}, b_{1} ]$.
Then for any $\lambda(x) \geq 0, \int_{ [ a_{1}, b_{1} ] \setminus ( a_{0}, b_{0} ) } \lambda(x) \mathrm{d}x = 1$ and any $w(x) \geq 0, \int_{ a_{0} }^{ b_{0} } w(x) \mathrm{d}x = 1$ such that
\begin{align}
\nonumber
\int_{ [ a_{1}, b_{1} ] \setminus ( a_{0}, b_{0} ) } \lambda(x) x \mathrm{d}x = \int_{ a_{0} }^{ b_{0} } w(x)x \mathrm{d}x = c,
\end{align}
we have
\begin{align}
\nonumber
\int_{ [ a_{1}, b_{1} ] \setminus ( a_{0}, b_{0} ) } \lambda(x) f(x) \mathrm{d}x \geq \int_{ a_{0} }^{ b_{0} } w(x) f(x) \mathrm{d}x.
\end{align}
If $f$ is strictly weighted convex at $c$ over $[ a_{1}, b_{1} ]$, then the inequality is strict.
\end{theorem}


\begin{theorem}\label{0711004}
Let $D \neq \emptyset$ be a subset of $\mathbb{R}$, and let $f$ be weighted convex at $c \in D^{*}$ over $D$.
Then for any random variables $\mu, \nu \in D$ such that
\begin{enumerate}[(1)]
\item $\operatorname{\bf E}\mu = \operatorname{\bf E}\nu = c$,
\item there are real numbers $m$ and $M$ such that $\mu \not\in (m, M)$ and $\nu \in [m, M]$,
\end{enumerate}
we have
\begin{align}
\nonumber
\operatorname{\bf E} f(\mu) \geq \operatorname{\bf E} f(\nu).
\end{align}
If $f$ is strictly weighted convex at $c$ over $D$, then the equality holds if and only if $\operatorname{P}( \mu = m, M ) = 1$ and $\operatorname{P}( \nu = m, M ) = 1$.
\end{theorem}

\section{The Jensen gap}

\begin{theorem}\label{0328016}
Let $f$ be continuous on $[a, b]$, and let $f^{\prime}$ be increasing on $(a, b)$.
Suppose that $f^{\prime}(pa + qb) = \frac{ f(b) - f(a) }{b - a}$, where $p, q > 0, p + q = 1$.
Then for any $x_{i} \in [a, b]$ and any $\lambda_{i} > 0, \sum_{i=1}^{n}\lambda_{i} = 1$, we have
\begin{align}
\nonumber
0 \leq \sum_{i=1}^{n}\lambda_{i} f(x_{i}) - f\left( \sum_{i=1}^{n}\lambda_{i}x_{i} \right) \leq p f(a) + q f(b) - f(pa + qb).
\end{align}
If $f^{\prime}$ is strictly increasing on $(a, b)$, then the left equality holds if and only if $x_{1} = \cdots = x_{n}$, and the right equality holds if and only if $\sum_{ x_{i} = a }\lambda_{i} = p$ and $\sum_{ x_{i} = b }\lambda_{i} = q$.
\end{theorem}
\begin{proof}Let $x_{i} = \theta_{i} a + ( 1 - \theta_{i} ) b, i = 1, \ldots, n$. By Theorem \ref{0406077},
\begin{align}\label{0719001}
\begin{split}
  &\left( \sum_{i=1}^{n}\lambda_{i}\theta_{i} \right) f(a) + \left( 1 - \sum_{i=1}^{n}\lambda_{i}\theta_{i} \right) f(b) - \sum_{i=1}^{n}\lambda_{i} f(x_{i}) \\
= &\sum_{i=1}^{n}\lambda_{i}\Big[ \theta_{i} f(a) + ( 1 - \theta_{i} ) f(b) - f( \theta_{i}a + ( 1 - \theta_{i} )b ) \Big].
\end{split}
\end{align}
Note that $\theta a + ( 1 - \theta ) b = \sum_{i=1}^{n}\lambda_{i} x_{i}$, where $\theta = \sum_{i=1}^{n}\lambda_{i}\theta_{i}$. This concludes the proof.
\end{proof}
\begin{flushleft}
$1^{0}$. S. Simi$\acute{ \text{c} }$ \cite{02} proved the theorem by using the following tricks:
\begin{align}
\nonumber
     &\sum_{i=1}^{n}\lambda_{i} f(x_{i}) - f\left( \sum_{i=1}^{n}\lambda_{i}x_{i} \right) \\
\nonumber
\leq &\sum_{i=1}^{n}\lambda_{i} \Big( \theta_{i} f(a) + ( 1 - \theta_{i} ) f(b) \Big) - f\left( \sum_{i=1}^{n}\lambda_{i}x_{i} \right) \\
\nonumber
=    &\theta f(a) + (1 - \theta) f(b) - f( \theta a + (1 - \theta) b ),
\end{align}
which depends on the convexity of $f$, while (\ref{0719001}) always holds regardless of the convexity.
\end{flushleft}


\begin{example}\label{0328005}
For any $x_{i} \in [a, b], \lambda_{i} > 0, \sum_{i=1}^{n}\lambda_{i} = 1$ and $p, q > 0, p + q = 1$ such that $\sum_{i=1}^{n}\lambda_{i}x_{i} = pa + qb$, we have
\begin{align}
\nonumber
0 \leq \sum_{i=1}^{n}\lambda_{i}x_{i}^{2} - \left( \sum_{i=1}^{n}\lambda_{i}x_{i} \right)^{2} = \sum_{1 \leq i < j \leq n}\lambda_{i}\lambda_{j}( x_{i} - x_{j} )^{2} \leq pq(b - a)^{2} \leq \left. (b-a)^{2} \right/ 4.
\end{align}
\end{example}


\begin{theorem}[Inequality of arithmetic and geometric means]\label{0328018}\hfill \\
Let $b > a > 0$, $b/a = 1 + r$ and $\rho = \frac{ \ln(1+r) }{r}$.
Then for any $x_{i} \in [a, b]$ and any $\lambda_{i} > 0, \sum_{i=1}^{n}\lambda_{i} = 1$, we have
\begin{align}
\nonumber
1 \leq \frac{ \lambda_{1}x_{1} + \cdots + \lambda_{n}x_{n} }{ x_{1}^{ \lambda_{1} } \cdots x_{n}^{ \lambda_{n} } } \leq \frac{ \mathrm{e}^{\rho-1} }{\rho}.
\end{align}
The left equality holds if and only if $x_{1} = \cdots = x_{n}$, and the right equality holds if and only if
$\sum_{ x_{i} = a }\lambda_{i} = 1 + \frac{1}{r} - \frac{1}{\ln(1+r)}$ and $\sum_{ x_{i} = b }\lambda_{i} = \frac{1}{\ln(1+r)} - \frac{1}{r}$.
\end{theorem}
\begin{flushleft}
$1^{0}$. $\frac{ \mathrm{e}^{\rho-1} }{\rho} < 1 + \frac{ r^{2} }{8}, r > 0$. \\
\begin{proof}Let $f(x) = \frac{ x(1+x)^{ \frac{1}{x} } }{ \mathrm{e}\ln(1+x) } - 1 - \frac{ x^{2} }{8}$.
If $x > 0$, then $(1+x)^{ \frac{1}{x} } < \mathrm{e}$, $0 < x -\ln(1+x) < \frac{ x^{2} }{2}$ and $\ln(1+x) - \frac{x}{1+x} > 0$, hence,
\begin{align}
\nonumber
x f^{\prime}(x) \ln^{2}(1+x) = &\frac{ (1+x)^{ \frac{1}{x} } }{ \mathrm{e} }( x - \ln(1+x) )\left( \ln(1+x) - \frac{x}{1+x} \right) - \frac{ x^{2}\ln^{2}(1+x) }{4} \\
\nonumber
< &\frac{ x^{2} }{2}\left( \ln(1+x) - \frac{x}{1+x} - \frac{ \ln^{2}(1+x) }{2} \right) := \frac{ x^{2} }{2} g(x).
\end{align}
Note that if $x > 0$, then $g^{\prime}(x) = -\frac{1}{1+x}\left( \ln(1+x) - \frac{x}{1+x} \right) < 0$ and hence $g(x) < g(0) = 0$.
If $x > 0$, then $f^{\prime}(x) < 0$ and hence $f(x) < \lim_{ x \rightarrow 0^{+} } f(x) = 0$.
\end{proof}
\end{flushleft}


\begin{corollary}\label{0604003}
Let $x_{i} > 0$ and $\lambda_{i} > 0, \sum_{i=1}^{n}\lambda_{i} = 1$. Let $\max_{i} x_{i} / \min_{i} x_{i} = 1 + r$. Then
\begin{align}
\nonumber
1 \leq \frac{ \lambda_{1}x_{1} + \cdots + \lambda_{n}x_{n} }{ x_{1}^{ \lambda_{1} } \cdots x_{n}^{ \lambda_{n} } } \leq 1 + \frac{ r^{2} }{8},
\end{align}
and the left equality (resp. right equality) holds if and only if $x_{1} = \cdots = x_{n}$.
\end{corollary}


\begin{theorem}\label{0715002}
Let $f$ be continuous on $[a, b]$, and let $f^{\prime}$ be increasing on $(a, b)$.
Suppose that $f^{\prime}(pa + qb) = \frac{ f(b) - f(a) }{b - a}$, where $p, q > 0, p + q = 1$.
Then for any $\lambda(x) \geq 0, \int_{a}^{b} \lambda(x) \mathrm{d}x = 1$, we have
\begin{align}
\nonumber
0 \leq \int_{a}^{b} \lambda(x) f(x) \mathrm{d}x - f\left( \int_{a}^{b}\lambda(x) x \mathrm{d}x \right) \leq p f(a) + q f(b) - f(pa + qb).
\end{align}
If $f^{\prime}$ is strictly increasing on $(a, b)$, then the inequalities are strict.
\end{theorem}
\begin{remark}
Further conclusions on the Jensen gap, see Theorems \ref{0209001} and \ref{0801001}.
\end{remark}

\section{The Hermite-Hadamard integral inequality}

\centerline{ \textbf{Part \uppercase\expandafter{\romannumeral1}} }

\begin{theorem}\label{0212005}
Let $f: [a, b] \rightarrow \mathbb{R}$ be convex and $x_{i} = a + i\frac{b - a}{n}, i = 0, 1, \ldots, n$, where $n \in \mathbb{N}$ and $n \geq 2$. Then
\begin{align}
\nonumber
f\left( \frac{a + b}{2} \right) \leq \frac{ f(x_{0}) + f(x_{1}) + \cdots + f(x_{n}) }{n+1} \leq \frac{ f(a) + f(b) }{2}.
\end{align}
If $f$ is strictly convex, then the inequalities are strict.
\end{theorem}
\begin{flushleft}
$1^{0}$. The theorem is a discrete version of the {\it Hermite-Hadamard integral inequality}.
(See Theorem \ref{0316002}.) \\
$2^{0}$. The theorem is a special case of Theorem \ref{0312001}.
\end{flushleft}


\begin{theorem}\label{0316002}
Let $f: [a, b] \rightarrow \mathbb{R}$ be convex. Then
\begin{align}
\nonumber
f\left( \frac{a + b}{2} \right) \leq \frac{1}{b - a}\int_{a}^{b} f(x) \mathrm{d}x \leq \frac{ f(a) + f(b) }{2}.
\end{align}
If $f$ is strictly convex, then the inequalities are strict.
\end{theorem}
\begin{flushleft}
$1^{0}$. The theorem is a special case of Theorems \ref{0316003} and \ref{0520002}. \\
$2^{0}$. This is the {\it Hermite-Hadamard integral inequality for convex functions}. For more details on the inequality, see \cite{05} and \cite{04}.
\end{flushleft}


\begin{theorem}\label{0312001}
Let $f: [a, b] \rightarrow \mathbb{R}$ be convex. Then for any
\begin{enumerate}[(1)]
\item $x_{i} \in [a, b]$ and $\lambda_{i} > 0, \sum_{i=1}^{n}\lambda_{i} = 1$,
\item $p, q > 0, p + q = 1$,
\end{enumerate}
such that $\lambda_{1}x_{1} + \cdots + \lambda_{n}x_{n} = pa + qb$, we have
\begin{align}
\nonumber
f(pa + qb) \leq \lambda_{1} f(x_{1}) + \cdots + \lambda_{n} f(x_{n}) \leq p f(a) + q f(b).
\end{align}
If $f$ is strictly convex, then the left equality holds if and only if $x_{1} = \cdots = x_{n}$, and the right equality holds if and only if $x_{i} = a\ \text{or}\ b, i = 1, \ldots, n$.
\end{theorem}
\begin{flushleft}
$1^{0}$. The theorem is a special case of Theorem \ref{0331005}.
\end{flushleft}


\begin{theorem}\label{0316003}
Let $f: [a, b] \rightarrow \mathbb{R}$ be convex.
Let $\lambda(x) \geq 0, \int_{a}^{b} \lambda(x) \mathrm{d}x = 1$ and $p, q > 0, p + q = 1$ such that $\int_{a}^{b} \lambda(x)x \mathrm{d}x = pa + qb$. Then
\begin{align}
\nonumber
f(pa + qb) \leq \int_{a}^{b} \lambda(x) f(x) \mathrm{d}x \leq p f(a) + q f(b).
\end{align}
If $f$ is strictly convex, then the inequalities are strict.
\end{theorem}
\begin{flushleft}
$1^{0}$. The theorem is a special case of Theorem \ref{0406055}.
\end{flushleft}


\begin{theorem}\label{0331003}
Let $D \neq \emptyset$ be a subset of $\mathbb{R}$, and let $f: D \rightarrow \mathbb{R}$ be convex. Then for any
\begin{enumerate}[(1)]
\item $x_{i} \in D$ and $\lambda_{i} > 0, \sum_{i=1}^{n}\lambda_{i} = 1$,
\item $a, b \in D$ and $p, q > 0, p + q = 1$,
\end{enumerate}
such that $\lambda_{1}x_{1} + \cdots + \lambda_{n}x_{n} = pa + qb \in D$, where $x_{i} \in [a, b]$, we have
\begin{align}
\nonumber
f(pa + qb) \leq \lambda_{1} f(x_{1}) + \cdots + \lambda_{n} f(x_{n}) \leq p f(a) + q f(b).
\end{align}
If $f$ is strictly convex, then the left equality holds if and only if $x_{1} = \cdots = x_{n}$, and the right equality holds if and only if $x_{i} = a\ \text{or}\ b, i = 1, \ldots, n$.
\end{theorem}
\begin{flushleft}
$1^{0}$. The theorem is a special case of Theorem \ref{0605004}.
\end{flushleft}


\begin{theorem}\label{0316004}
Let $D \neq \emptyset$ be a subset of $\mathbb{R}$, and let $f: D \rightarrow \mathbb{R}$ be convex. Then for any
\begin{enumerate}[(1)]
\item random variable $\nu \in D$,
\item $a, b \in D$ and $p, q > 0, p + q = 1$,
\end{enumerate}
such that $\operatorname{\bf E}\nu = pa + qb \in D$, where $\nu \in [a, b]$, we have
\begin{align}
\nonumber
f(pa + qb) \leq \operatorname{\bf E} f(\nu) \leq p f(a) + q f(b).
\end{align}
If $f$ is strictly convex, then the left equality holds if and only if $\operatorname{P}( pa + qb) = 1$, and the right equality holds if and only if $\operatorname{P}(\nu = a, b) = 1$.
\end{theorem}
\begin{flushleft}
$1^{0}$. The theorem is a special case of Theorem \ref{0406048}.
\end{flushleft}

\centerline{ \textbf{Part \uppercase\expandafter{\romannumeral2}} }

\begin{flushleft}
S. Simi$\acute{ \text{c} }$ proved the following theorem in \cite{01}.
\end{flushleft}
\begin{theorem}\label{0308001}
If $f: [a, b] \rightarrow \mathbb{R}$ is convex, then
\begin{align}
\nonumber
0 \leq \frac{1}{b - a}\int_{a}^{b} f(x) \mathrm{d}x - f\left( \frac{1}{b - a}\int_{a}^{b} x \mathrm{d}x \right) \leq \frac{1}{2}\left[ \frac{ f(a) + f(b) }{2} - f\left( \frac{a + b}{2} \right) \right].
\end{align}
If $f$ is strictly convex, then the inequalities are strict.
\end{theorem}
\begin{flushleft}
$1^{0}$. The constant $1/2$ is optimal. (See Example \ref{0308002}.) \\
$2^{0}$. The theorem is a special case of Theorem \ref{0327007}.
\end{flushleft}


\begin{example}\label{0308002}
If $f(x) = \vert x \vert$, then
\begin{align}
\nonumber
  &\frac{1}{ 1 - (-1) }\int_{-1}^{1} f(x) \mathrm{d}x - f\left( \frac{1}{ 1 - (-1) }\int_{-1}^{1} x \mathrm{d}x \right) \\
\nonumber
= &\frac{1}{2}\left[ \frac{ f(-1) + f(1) }{2} - f\left( \frac{-1 + 1}{2} \right) \right] = \frac{1}{2}.
\end{align}
\end{example}
\begin{flushleft}
$1^{0}$. An open question was raised in \cite{01}: find the best possible bound $N( \alpha^{*}, \beta^{*} )$ such that
\begin{align}
\nonumber
\frac{1}{b - a}\int_{a}^{b} g(x) \mathrm{d}x \leq \alpha^{*} ( g(a) + g(b) ) + \beta^{*} g\left( \frac{a + b}{2} \right) := N( \alpha^{*}, \beta^{*} )
\end{align}
holds for every convex function $g \in C^{\infty}[a, b]$.

Let
\begin{align}
\nonumber
f_{n}(x) = \left\{\begin{array}{cc}
\vert x \vert\exp\left( \frac{ 1 - 1/\vert x \vert }{n} \right), & x \neq 0, \\
0, & x = 0.
\end{array}\right.
\end{align}
Then for every $n \in \mathbb{N}_{+}$, we have
\begin{enumerate}[(1)]
\item $f_{n}$ is strictly convex,
\item $f_{n}$ has derivatives of all orders,
\item $f_{n}( \pm 1 ) = f( \pm 1 ) = 1$ and $f_{n}(0) = f(0) = 0$.
\end{enumerate}
There are $\alpha_{n}, \beta_{n} > 0, 2\alpha_{n} + \beta_{n} = 1$ such that
\begin{align}
\nonumber
\frac{1}{ 1 - (-1) }\int_{-1}^{1} f_{n}(x) \mathrm{d}x = \alpha_{n} ( f_{n}(-1) + f_{n}(1) ) + \beta_{n} f_{n}\left( \frac{ -1 + 1 }{2} \right).
\end{align}
Note that $f_{n}(x) \leq f_{n+1}(x)$ and $\lim_{ n \rightarrow +\infty } f_{n}(x) = f(x)$ hold for all $x \in [-1, 1]$. Then
\begin{align}
\nonumber
\lim_{ n \rightarrow +\infty }\int_{-1}^{1} f_{n}(x) \mathrm{d}x = \int_{-1}^{1} f(x) \mathrm{d}x,
\end{align}
and then $\lim_{ n \rightarrow +\infty }\alpha_{n} = 1/4$ and $\lim_{ n \rightarrow +\infty }\beta_{n} = 1/2$.

Therefore, $N( \alpha^{*}, \beta^{*} ) = N(1/4, 1/2) = \frac{ g(a) + g(b) }{4} + \frac{1}{2} g\left( \frac{a + b}{2} \right)$.
\end{flushleft}


\begin{theorem}\label{0327007}
Let $f: [a, b] \rightarrow \mathbb{R}$ be convex. Then for any $p, q > 0, p + q = 1$, we have
\begin{align}
\nonumber
f\left( \frac{a + b}{2} \right) \leq \frac{1}{b - a}\int_{a}^{b} f(x) \mathrm{d}x \leq \frac{ f(a) + f(b) }{2} - \frac{ p f(a) + q f(b) - f(pa + qb) }{2}.
\end{align}
If $f$ is strictly convex, then the inequalities are strict.
\end{theorem}
\begin{proof}Let $c = pa + qb$. By Theorem \ref{0316002},
\begin{align}
\nonumber
\frac{1}{c - a}\int_{a}^{c} f(x) \mathrm{d}x \leq \frac{ f(a) + f(c) }{2}\ \text{and}\ \frac{1}{b - c}\int_{c}^{b} f(x) \mathrm{d}x \leq \frac{ f(c) + f(b) }{2}.
\end{align}
Then
\begin{align}
\nonumber
\frac{1}{b - a}\int_{a}^{b} f(x) \mathrm{d}x \leq &\frac{1}{b - a}\left( (c - a)\frac{ f(a) + f(c) }{2} + (b - c)\frac{ f(c) + f(b) }{2} \right) \\
\nonumber
= &\frac{ f(a) + f(b) }{2} - \frac{ p f(a) + q f(b) - f(pa + qb) }{2}.
\end{align}
This concludes the proof.
\end{proof}
\begin{flushleft}
$1^{0}$. The theorem is a special case of Theorem \ref{0423007}.
\end{flushleft}


\begin{theorem}\label{0423007}
Let $f: [a, b] \rightarrow \mathbb{R}$ be convex and $\lambda(x) \geq 0, \int_{a}^{b} \lambda(x) \mathrm{d}x = 1$.
Suppose that $P = \{ c_{0}, c_{1}, \ldots, c_{n} \}$ is a partition of $[a, b]$.
Let $w_{0} = R_{0}, w_{1} = L_{1} + R_{1}, \ldots, w_{n-1} = L_{n-1} + R_{n-1}, w_{n} = L_{n}$, where
\begin{align}
\nonumber
\left\{\begin{array}{c}
R_{i-1} + L_{i} = \int_{ c_{i-1} }^{ c_{i} } \lambda(x) \mathrm{d}x, \\
R_{i-1} c_{i-1} + L_{i} c_{i} = \int_{ c_{i-1} }^{ c_{i} } \lambda(x)x \mathrm{d}x,
\end{array}\right. i = 1, \ldots, n.
\end{align}
Then
\begin{align}
\nonumber
f\left( \sum_{i=0}^{n} w_{i}c_{i} \right) \leq \int_{a}^{b} \lambda(x) f(x) \mathrm{d}x \leq \sum_{i=0}^{n} w_{i} f(c_{i}).
\end{align}
If $f$ is strictly convex, then the inequalities are strict.
\end{theorem}
\begin{flushleft}
$1^{0}$. $w_{i} \geq 0, \sum_{i=0}^{n} w_{i} = 1$, and $w_{i}$ do not depend on the convex function $f$. \\
$2^{0}$. $\sum_{i=0}^{n} w_{i}c_{i} = \int_{a}^{b} \lambda(x)x \mathrm{d}x$.
\end{flushleft}


\begin{example}\label{0424002}
If $f: [0, 1] \rightarrow \mathbb{R}$ is convex, then
\begin{align}
\nonumber
\int_{0}^{1} f(x) \mathrm{d}x \leq \frac{1}{14} f(0) + \frac{2}{7} f\left( \frac{1}{7} \right) + \frac{3}{7} f\left( \frac{4}{7} \right) + \frac{3}{14} f(1).
\end{align}
If $f$ is strictly convex, then the inequality is strict.
\end{example}
\begin{proof}By Theorem \ref{0423007}, we conclude the proof.
\end{proof}


\begin{theorem}\label{0424001}
Let $f: [a, b] \rightarrow \mathbb{R}$ be convex, and let $x_{i} \in [a, b]$ and $\lambda_{i} > 0, \sum_{i=1}^{n}\lambda_{i} = 1$.
Suppose that $P = \{ c_{0}, c_{1}, \ldots, c_{m} \}$ is a partition of $[a, b]$.
Let $\mathcal{I}_{1}, \ldots, \mathcal{I}_{m}$ be subsets of $\{ 1, \ldots, n \}$ such that
\begin{enumerate}[(1)]
\item $x_{s} \in [ c_{i-1}, c_{i} ], s \in \mathcal{I}_{i}, i = 1, \ldots, m$,
\item $\mathcal{I}_{i} \cap \mathcal{I}_{j} = \emptyset$ for all $i \neq j$,
\item $\mathcal{I}_{1} \cup \cdots \cup \mathcal{I}_{m} = \{ 1, \ldots, n \}$.
\end{enumerate}
Let $w_{0} = R_{0}, w_{1} = L_{1} + R_{1}, \ldots, w_{m-1} = L_{m-1} + R_{m-1}, w_{m} = L_{m}$, where
\begin{align}
\nonumber
\left\{\begin{array}{c}
R_{i-1} + L_{i} = \sum_{ s \in \mathcal{I}_{i} } \lambda_{s}, \\
R_{i-1} c_{i-1} + L_{i} c_{i} = \sum_{ s \in \mathcal{I}_{i} } \lambda_{s} x_{s},
\end{array}\right. i = 1, \ldots, m.
\end{align}
Then
\begin{align}
\nonumber
f\left( \sum_{i=0}^{m} w_{i}c_{i} \right) \leq \sum_{i=1}^{n}\lambda_{i} f(x_{i}) \leq \sum_{i=0}^{m} w_{i} f(c_{i}).
\end{align}
If $f$ is strictly convex, then the left equality holds if and only if $x_{1} = \cdots = x_{n}$, and the right equality holds if and only if $x_{1}, \ldots, x_{n} \in \{ c_{0}, c_{1}, \ldots, c_{m} \}$.
\end{theorem}
\begin{flushleft}
$1^{0}$. If $\mathcal{I}_{1}, \ldots, \mathcal{I}_{m} \neq \emptyset$, then $\mathcal{I}_{1}, \ldots, \mathcal{I}_{m}$ is a partition of $\{ 1, \ldots, n \}$. \\
$2^{0}$. $w_{i} \geq 0, \sum_{i=0}^{m} w_{i} = 1$, and $w_{i}$ do not depend on the convex function $f$. \\
$3^{0}$. $\sum_{i=0}^{m} w_{i}c_{i} = \sum_{i=1}^{n}\lambda_{i}x_{i}$.
\end{flushleft}


\begin{example}\label{0424003}
If $f: [0, 1] \rightarrow \mathbb{R}$ is convex, then
\begin{align}
\nonumber
\frac{ f(0) + f\left( \frac{1}{17} \right) + \cdots + f\left( \frac{16}{17} \right) + f(1) }{18} \leq \frac{5}{51} f(0) + \frac{245}{918} f\left( \frac{1}{7} \right) + \frac{371}{918} f\left( \frac{4}{7} \right) + \frac{106}{459} f(1).
\end{align}
If $f$ is strictly convex, then the inequality is strict.
\end{example}
\begin{proof}By Theorem \ref{0424001}, we conclude the proof.
\end{proof}

\centerline{ \textbf{Part \uppercase\expandafter{\romannumeral3}} }

\begin{theorem}\label{0423001}
If $f: [0, 1] \rightarrow \mathbb{R}$ is convex, then
\begin{align}
\nonumber
\frac{ f\left( \frac{1}{n+1} \right) + \cdots + f\left( \frac{n}{n+1} \right) }{n} \geq \frac{ f\left( \frac{1}{n} \right) + \cdots + f\left( \frac{n-1}{n} \right) }{n-1}, n = 2, 3, \ldots.
\end{align}
If $f$ is strictly convex, then the inequalities are strict.
\end{theorem}
\begin{proof}We have
\begin{align}
\nonumber
     &\frac{1}{n}\sum_{i=1}^{n} f\left( \frac{i}{n+1} \right) = \sum_{i=1}^{n-1}\left[ \left( \frac{i}{n} - \frac{i-1}{n-1} \right) f\left( \frac{i}{n+1} \right) + \left( \frac{i}{n-1} - \frac{i}{n} \right) f\left( \frac{i+1}{n+1} \right) \right] \\
\nonumber
\geq &\sum_{i=1}^{n-1}\left[ \left( \frac{i}{n} - \frac{i-1}{n-1} \right) + \left( \frac{i}{n-1} - \frac{i}{n} \right) \right] f\left( \frac{ \left( \frac{i}{n} - \frac{i-1}{n-1} \right)\frac{i}{n+1} + \left( \frac{i}{n-1} - \frac{i}{n} \right)\frac{i+1}{n+1} }{ \left( \frac{i}{n} - \frac{i-1}{n-1} \right) + \left( \frac{i}{n-1} - \frac{i}{n} \right) } \right) \\
\nonumber
=    &\frac{1}{n-1}\sum_{i=1}^{n-1} f\left( \frac{i}{n} \right).
\end{align}
This concludes the proof.
\end{proof}


\begin{corollary}\label{0423002}
Let $f: [a, b] \rightarrow \mathbb{R}$ be (strictly) convex.
Then $\left\{ \frac{ f(x_{1}) + \cdots + f(x_{n-1}) }{n-1} \right\}$ is (strictly) increasing and converges to $\frac{1}{b - a}\int_{a}^{b} f(x) \mathrm{d}x$, where $x_{i} = a + i\frac{b - a}{n}$.
\end{corollary}
\begin{proof}By Theorem \ref{0423001}, we conclude the proof.
\end{proof}
\begin{flushleft}
$1^{0}$. We provide another method to prove $\frac{1}{b - a}\int_{a}^{b} f(x) \mathrm{d}x \geq \frac{ f(x_{1}) + \cdots + f(x_{n-1}) }{n-1}, n = 2, 3, \ldots$.
\begin{proof}Let $y_{i} = \frac{ x_{i-1} + x_{i} }{2}, i = 1, \ldots, n$. Then
\begin{align}
\nonumber
\frac{1}{b - a}\int_{ y_{1} }^{ y_{n} } f(x) \mathrm{d}x = \frac{1}{n}\sum_{i=1}^{n-1}\frac{1}{ y_{i+1} - y_{i} }\int_{ y_{i} }^{ y_{i+1} } f(x) \mathrm{d}x \geq \frac{ f(x_{1}) + \cdots + f(x_{n-1}) }{n}.
\end{align}
Note that
\begin{enumerate}[(1)]
\item $x_{1}, \ldots, x_{n-1} \in [ y_{1}, y_{n} ]$,
\item $\frac{1}{b - a}\int_{ [a, b] \setminus ( y_{1}, y_{n} ) } \mathrm{d}x = \left( \frac{1}{n-1} - \frac{1}{n} \right)(n-1)$,
\item $\frac{1}{b - a}\int_{ [a, b] \setminus ( y_{1}, y_{n} ) } x \mathrm{d}x = \left( \frac{1}{n-1} - \frac{1}{n} \right)( x_{1} + \cdots + x_{n-1} )$.
\end{enumerate}
By Theorem \ref{0406081}, $\frac{1}{b - a}\int_{ [a, b] \setminus ( y_{1}, y_{n} ) } f(x) \mathrm{d}x \geq \left( \frac{1}{n-1} - \frac{1}{n} \right)( f(x_{1}) + \cdots + f(x_{n-1}) )$. Then
\begin{align}
\nonumber
\frac{1}{b - a}\int_{a}^{b} f(x) \mathrm{d}x = &\frac{1}{b - a}\int_{ y_{1} }^{ y_{n} } f(x) \mathrm{d}x + \frac{1}{b - a}\int_{ [a, b] \setminus ( y_{1}, y_{n} ) } f(x) \mathrm{d}x \\
\nonumber
\geq &\frac{ f(x_{1}) + \cdots + f(x_{n-1}) }{n-1}.
\end{align}
This concludes the proof.
\end{proof}
\end{flushleft}


\begin{theorem}\label{0429001}
If $f: [a, b] \rightarrow \mathbb{R}$ is convex, then
\begin{align}
\nonumber
\frac{ f\left( \frac{0}{n} \right) + f\left( \frac{1}{n} \right) + \cdots + f\left( \frac{n}{n} \right) }{n+1} \geq \frac{ f\left( \frac{0}{n+1} \right) + f\left( \frac{1}{n+1} \right) + \cdots + f\left( \frac{n+1}{n+1} \right) }{n+2}, n = 1, 2, \ldots.
\end{align}
If $f$ is strictly convex, then the inequalities are strict.
\end{theorem}
\begin{proof}We have
\begin{align}
\nonumber
\frac{1}{n+1}\sum_{i=0}^{n} f\left( \frac{i}{n} \right) = &\frac{ f(0) + f(1) }{n+2} + \frac{1}{n+2}\sum_{i=1}^{n}\left[ \frac{i}{n+1} f\left( \frac{i-1}{n} \right) + \frac{n-i+1}{n+1} f\left( \frac{i}{n} \right) \right] \\
\nonumber
\geq &\frac{ f(0) + f(1) }{n+2} + \frac{1}{n+2}\sum_{i=1}^{n} f\left( \frac{i}{n+1}\cdot\frac{i-1}{n} + \frac{n-i+1}{n+1}\cdot\frac{i}{n} \right) \\
\nonumber
= & \frac{1}{n+2}\sum_{i=0}^{n+1} f\left( \frac{i}{n+1} \right).
\end{align}
This concludes the proof.
\end{proof}


\begin{corollary}\label{0522001}
Let $f: [a, b] \rightarrow \mathbb{R}$ be (strictly) convex.
Then $\left\{ \frac{ f(x_{0}) + f(x_{1}) + \cdots + f(x_{n}) }{n+1} \right\}$ is (strictly) decreasing and converges to $\frac{1}{b - a}\int_{a}^{b} f(x) \mathrm{d}x$, where $x_{i} = a + i\frac{b - a}{n}$.
\end{corollary}
\begin{proof}By Theorem \ref{0429001}, we conclude the proof.
\end{proof}
\begin{flushleft}
$1^{0}$. We provide another method to prove $\frac{1}{b - a}\int_{a}^{b} f(x) \mathrm{d}x \leq \frac{ f(x_{0}) + f(x_{1}) + \cdots + f(x_{n}) }{n+1}, n = 2, 3, \ldots$. \\
\begin{proof}By Theorems \ref{0312001} and \ref{0423007},
\begin{align}
\nonumber
  &\frac{ f(x_{0}) + f(x_{1}) + \cdots + f(x_{n}) }{n+1} - \frac{1}{b - a}\int_{a}^{b} f(x) \mathrm{d}x \\
\nonumber
\geq &\frac{ f(x_{0}) + f(x_{1}) + \cdots + f(x_{n}) }{n+1} - \left( \frac{ f(x_{0}) }{2n} + \frac{ f(x_{1}) + \cdots + f(x_{n-1}) }{n} + \frac{ f(x_{n}) }{2n} \right) \\
\nonumber
= &\frac{n-1}{ n(n+1) }\left( \frac{ f(x_{0}) + f(x_{n}) }{2} - \frac{ f(x_{1}) + \cdots + f(x_{n-1}) }{n-1} \right) \geq 0.
\end{align}
This concludes the proof.
\end{proof}
\end{flushleft}


\begin{theorem}\label{0423004}
Let $f: [a, b] \rightarrow \mathbb{R}$ be convex, and let $x_{i} = a + i\frac{b - a}{n}$, where $n \in \mathbb{N}$ and $n \geq 2$. Then
\begin{align}
\nonumber
\frac{ f(x_{1}) + \cdots + f(x_{n-1}) }{n-1} \leq \frac{1}{b - a}\int_{a}^{b} f(x) \mathrm{d}x \leq \frac{ f(x_{0}) + f(x_{1}) + \cdots + f(x_{n}) }{n+1}.
\end{align}
If $f$ is strictly convex, then the inequalities are strict.
\end{theorem}

\section{The squeeze theorem for convex functions}

\begin{theorem}\label{0331005}
Let $f: [ a_{1}, b_{1} ] \setminus ( a_{0}, b_{0} ) \rightarrow \mathbb{R}$ be convex ($[ a_{0}, b_{0} ] \subseteq [ a_{1}, b_{1} ]$).
Then for any
\begin{enumerate}[(1)]
\item $x_{i} \in [ a_{1}, b_{1} ] \setminus ( a_{0}, b_{0} )$ and $\lambda_{i} > 0, \sum_{i=1}^{n}\lambda_{i} = 1$,
\item $p_{s}, q_{s} > 0, p_{s} + q_{s} = 1$,
\end{enumerate}
such that $p_{0}a_{0} + q_{0}b_{0} = \lambda_{1}x_{1} + \cdots + \lambda_{n}x_{n} = p_{1}a_{1} + q_{1}b_{1}$, we have
\begin{align}
\nonumber
p_{0} f(a_{0}) + q_{0} f(b_{0}) \leq \lambda_{1} f(x_{1}) + \cdots + \lambda_{n} f(x_{n}) \leq p_{1} f(a_{1}) + q_{1} f(b_{1}).
\end{align}
If $f$ is strictly convex, then the left equality holds if and only if $x_{i} = a_{0}\ \text{or}\ b_{0}$, and the right equality holds if and only if $x_{i} = a_{1}\ \text{or}\ b_{1}$, $i = 1, \ldots, n$.
\end{theorem}
\begin{proof}By Theorems \ref{0227002} and \ref{0716001}, we conclude the proof.
\end{proof}


\begin{theorem}[Inequality of arithmetic and geometric means]\label{0328020}\hfill \\
Suppose that $[a_{0}, b_{0}] \subseteq [a_{1}, b_{1}] \subset (0, +\infty)$. Then for any
\begin{enumerate}[(1)]
\item $x_{i} \in [a_{1}, b_{1}] \setminus (a_{0}, b_{0})$ and $\lambda_{i} > 0, \sum_{i=1}^{n}\lambda_{i} = 1$,
\item $p_{s}, q_{s} > 0, p_{s} + q_{s} = 1$,
\end{enumerate}
such that $p_{0}a_{0} + q_{0}b_{0} = \lambda_{1}x_{1} + \cdots + \lambda_{n}x_{n} = p_{1}a_{1} + q_{1}b_{1}$, we have
\begin{align}
\nonumber
\frac{ p_{0}a_{0} + q_{0}b_{0} }{ a_{0}^{ p_{0} } b_{0}^{ q_{0} } } \leq \frac{ \lambda_{1}x_{1} + \cdots + \lambda_{n}x_{n} }{ x_{1}^{ \lambda_{1} } \cdots x_{n}^{ \lambda_{n} } } \leq \frac{ p_{1}a_{1} + q_{1}b_{1} }{ a_{1}^{ p_{1} } b_{1}^{ q_{1} } }.
\end{align}
The left equality holds if and only if $x_{i} = a_{0}\ \text{or}\ b_{0}$, and the right equality holds if and only if $x_{i} = a_{1}\ \text{or}\ b_{1}$, $i = 1, \ldots, n$.
\end{theorem}


\begin{theorem}\label{0520002}
Let $f: [ a_{1}, b_{1} ] \setminus ( a_{0}, b_{0} ) \rightarrow \mathbb{R}$ be convex ($[ a_{0}, b_{0} ] \subset ( a_{1}, b_{1} )$).
If $\frac{ a_{0} + b_{0} }{2} = \frac{ a_{1} + b_{1} }{2}$, then
\begin{align}
\nonumber
\frac{ f(a_{0}) + f(b_{0}) }{2} \leq \frac{1}{ (b_{1} - a_{1}) - (b_{0} - a_{0}) }\int_{ [ a_{1}, b_{1} ] \setminus ( a_{0}, b_{0} ) } f(x) \mathrm{d}x \leq \frac{ f(a_{1}) + f(b_{1}) }{2}.
\end{align}
If $f$ is strictly convex, then the inequalities are strict.
\end{theorem}
\begin{flushleft}
$1^{0}$ The theorem is a special case of Theorem \ref{0406055}.
\end{flushleft}


\begin{example}\label{0222007}
If $f: [0, 1] \setminus (1/3, 2/3) \rightarrow \mathbb{R}$ is convex, then
\begin{align}
\nonumber
\frac{ f(1/3) + f(2/3) }{3} \leq \int_{ [0, 1] \setminus (1/3, 2/3) } f(x) \mathrm{d}x \leq \frac{ f(0) + f(1) }{3}.
\end{align}
If $f$ is strictly convex, then the inequalities are strict.
\end{example}


\begin{theorem}\label{0406055}
Let $f: [ a_{1}, b_{1} ] \setminus ( a_{0}, b_{0} ) \rightarrow \mathbb{R}$ be convex ($[ a_{0}, b_{0} ] \subset ( a_{1}, b_{1} )$).
Then for any $\lambda(x) \geq 0, \int_{ [ a_{1}, b_{1} ] \setminus ( a_{0}, b_{0} ) } \lambda(x) \mathrm{d}x = 1$ and any $p_{s}, q_{s} > 0, p_{s} + q_{s} = 1$ such that
\begin{align}
\nonumber
p_{0}a_{0} + q_{0}b_{0} = \int_{ [ a_{1}, b_{1} ] \setminus ( a_{0}, b_{0} ) } \lambda(x) x \mathrm{d}x = p_{1}a_{1} + q_{1}b_{1},
\end{align}
we have
\begin{align}
\nonumber
p_{0} f(a_{0}) + q_{0} f(b_{0}) \leq \int_{ [ a_{1}, b_{1} ] \setminus ( a_{0}, b_{0} ) } \lambda(x) f(x) \mathrm{d}x \leq p_{1} f(a_{1}) + q_{1} f(b_{1}).
\end{align}
If $f$ is strictly convex, then the inequalities are strict.
\end{theorem}


\begin{theorem}\label{0605004}
Let $D \neq \emptyset$ be a subset of $\mathbb{R}$, and let $f: D \rightarrow \mathbb{R}$ be convex. Then for any
\begin{enumerate}[(1)]
\item $x_{i} \in D$ and $\lambda_{i} > 0, \sum_{i=1}^{n}\lambda_{i} = 1$,
\item $a_{s}, b_{s} \in D$ and $p_{s}, q_{s} > 0, p_{s} + q_{s} = 1$, where $[ a_{0}, b_{0} ] \subseteq [ a_{1}, b_{1} ]$,
\end{enumerate}
such that $p_{0}a_{0} + q_{0}b_{0} = \lambda_{1}x_{1} + \cdots + \lambda_{n}x_{n} = p_{1}a_{1} + q_{1}b_{1}, x_{i} \in [ a_{1}, b_{1} ] \setminus ( a_{0}, b_{0} )$, we have
\begin{align}
\nonumber
p_{0} f(a_{0}) + q_{0} f(b_{0}) \leq \lambda_{1} f(x_{1}) + \cdots + \lambda_{n} f(x_{n}) \leq p_{1} f(a_{1}) + q_{1} f(b_{1}).
\end{align}
If $f$ is strictly convex, then the left equality holds if and only if $x_{i} = a_{0}\ \text{or}\ b_{0}$, and the right equality holds if and only if $x_{i} = a_{1}\ \text{or}\ b_{1}$, $i = 1, \ldots, n$.
\end{theorem}


\begin{theorem}\label{0406048}
Let $D \neq \emptyset$ be a subset of $\mathbb{R}$, and let $f: D \rightarrow \mathbb{R}$ be convex. Then for any
\begin{enumerate}[(1)]
\item random variable $\nu \in D$,
\item $a_{s}, b_{s} \in D$ and $p_{s}, q_{s} > 0, p_{s} + q_{s} = 1$, where $[ a_{0}, b_{0} ] \subseteq [ a_{1}, b_{1} ]$,
\end{enumerate}
such that $p_{0}a_{0} + q_{0}b_{0} = \operatorname{\bf E} \nu = p_{1}a_{1} + q_{1}b_{1}, \nu \in [ a_{1}, b_{1} ] \setminus ( a_{0}, b_{0} )$, we have
\begin{align}
\nonumber
p_{0}f(a_{0}) + q_{0}f(b_{0}) \leq \operatorname{\bf E}f(\nu) \leq p_{1}f(a_{1}) + q_{1}f(b_{1}).
\end{align}
If $f$ is strictly convex, then the left equality holds if and only if $\operatorname{P}( \nu = a_{0}, b_{0} ) = 1$, and the right equality holds if and only if $\operatorname{P}( \nu = a_{1}, b_{1} ) = 1$.
\end{theorem}

\section{Generalizations of the Jensen's inequality}

\centerline{ \textbf{Part \uppercase\expandafter{\romannumeral1}} }

\begin{theorem}\label{0216002}
Let $I$ be an interval, and let $f, g: I \rightarrow \mathbb{R}$.

If for any $x, y \in I$ and any $p, q > 0, p + q = 1$, we have
\begin{align}
\nonumber
p f(x) + q f(y) + g(px + qy) \geq 0\ ( \text{resp}. > 0 ),
\end{align}
then for any $x_{i} \in I$ and any $\lambda_{i} > 0, \sum_{i=1}^{n}\lambda_{i} = 1$, we have
\begin{align}
\nonumber
\sum_{i=1}^{n}\lambda_{i} f(x_{i}) + g\left( \sum_{i=1}^{n}\lambda_{i}x_{i} \right) \geq 0\ ( \text{resp}. > 0 ).
\end{align}
\end{theorem}
\begin{flushleft}
$1^{0}$. If we take $g = -f$, then the {\it Jensen's inequality}. \\
$2^{0}$. The theorem is a special case of Theorem \ref{0209009}.
\end{flushleft}


\begin{theorem}\label{0209009}
Let $I, J$ be intervals, and let $f, g: I \rightarrow \mathbb{R}$.

If for any $x, y \in I$ and any $p, q > 0, p + q = 1$, we have
\begin{align}
\nonumber
p f(x) + q f(y) + g(px + qy) \in J,
\end{align}
then for any $x_{i} \in I$ and any $\lambda_{i} > 0, \sum_{i=1}^{n}\lambda_{i} = 1$, we have
\begin{align}
\nonumber
\sum_{i=1}^{n}\lambda_{i} f(x_{i}) + g\left( \sum_{i=1}^{n}\lambda_{i}x_{i} \right) \in J.
\end{align}
\end{theorem}
\begin{flushleft}
$1^{0}$. If we take $g = -f$, then $0 \in J$. \\
$2^{0}$. If we take $g = -f$ and $J = [0, +\infty)$ (resp. $J = (-\infty, 0]$), then the {\it Jensen's inequality}.
\end{flushleft}


\begin{example}\label{0406008}
For any $x, y \in \mathbb{R}$ and any $p, q > 0, p + q = 1$, we have
\begin{align}
\nonumber
-1 < p\lfloor x \rfloor + q\lfloor y \rfloor - \lfloor px + qy \rfloor < 1,
\end{align}
therefore for any $x_{i} \in \mathbb{R}$ and any $\lambda_{i} > 0, \sum_{i=1}^{n}\lambda_{i} = 1$, we have
\begin{align}
\nonumber
-1 < \sum_{i=1}^{n}\lambda_{i}\lfloor x_{i} \rfloor - \left\lfloor \sum_{i=1}^{n}\lambda_{i}x_{i} \right\rfloor < 1.
\end{align}
\end{example}
\begin{proof}Note that $x-1 < \lfloor x \rfloor \leq x$, $y-1 < \lfloor y \rfloor \leq y$, and $px + qy - 1 < \lfloor px + qy \rfloor \leq px + qy$.
Then $-1 < p\lfloor x \rfloor + q\lfloor y \rfloor - \lfloor px + qy \rfloor < 1$.
\end{proof}


\begin{example}\label{0406010}
For any $x, y \in [-\pi/2, \pi/2]$ and any $p, q > 0, p + q = 1$, we have
\begin{align}
\nonumber
-\frac{47}{170} < p\sin x + q\sin y - \sin(px + qy) < \frac{47}{170},
\end{align}
therefore for any $x_{i} \in [-\pi/2, \pi/2]$ and any $\lambda_{i} > 0, \sum_{i=1}^{n}\lambda_{i} = 1$, we have
\begin{align}
\nonumber
-\frac{47}{170} < \sum_{i=1}^{n}\lambda_{i}\sin x_{i} - \sin\left( \sum_{i=1}^{n}\lambda_{i}x_{i} \right) < \frac{47}{170}.
\end{align}
\end{example}


\begin{theorem}\label{0215009}
Let $f \in \mathcal{R}[a, b]$, and let $g: [a, b] \rightarrow \mathbb{R}$. Let $J$ be an interval.

If for any $x, y \in [a, b]$ and any $p, q > 0, p + q = 1$, we have
\begin{align}
\nonumber
p f(x) + q f(y) + g(px + qy) \in J,
\end{align}
then for any $\lambda(x) \geq 0, \int_{a}^{b} \lambda(x) \mathrm{d}x = 1$, we have
\begin{align}
\nonumber
\int_{a}^{b} \lambda(x) f(x) \mathrm{d}x + g\left( \int_{a}^{b} \lambda(x)x \mathrm{d}x \right) \in J.
\end{align}
\end{theorem}


\begin{theorem}\label{0709002}
Let $A, B \neq \emptyset$ be subsets of $\mathbb{R}$, and $f: A \rightarrow \mathbb{R}$ and $g: B \rightarrow \mathbb{R}$. Let $J$ be an interval.

If for any $x, y \in A$ and any $p, q > 0, p + q = 1$ such that $px + qy \in B$, we have
\begin{align}
\nonumber
p f(x) + q f(y) + g(px + qy) \in J,
\end{align}
then for any $x_{i} \in A$ and any $\lambda_{i} > 0, \sum_{i=1}^{n}\lambda_{i} = 1$ such that $\sum_{i=1}^{n}\lambda_{i}x_{i} \in B$, we have
\begin{align}
\nonumber
\sum_{i=1}^{n}\lambda_{i} f(x_{i}) + g\left( \sum_{i=1}^{n}\lambda_{i}x_{i} \right) \in J.
\end{align}
\end{theorem}
\begin{proof}By (\ref{0103001}), we conclude the proof.
\end{proof}


\begin{theorem}\label{0702003}
Let $f \in \mathcal{R}[a, b]$, and $B \neq \emptyset$ be a subset of $\mathbb{R}$ and $g: B \rightarrow \mathbb{R}$. Let $J$ be an interval.

If for any $x, y \in [a, b]$ and any $p, q > 0, p + q = 1$ such that $px + qy \in B$, we have
\begin{align}
\nonumber
p f(x) + q f(y) + g(px + qy) \in J,
\end{align}
then for any $\lambda(x) \geq 0, \int_{a}^{b} \lambda(x) \mathrm{d}x = 1$ such that $\int_{a}^{b} \lambda(x) x \mathrm{d}x \in B$, we have
\begin{align}
\nonumber
\int_{a}^{b} \lambda(x) f(x) \mathrm{d}x + g\left( \int_{a}^{b} \lambda(x)x \mathrm{d}x \right) \in J.
\end{align}
\end{theorem}
\begin{flushleft}
$1^{0}$. The theorem is equivalent to Theorem \ref{0606002}.
\end{flushleft}


\begin{theorem}\label{0702001}
Let $A, B \neq \emptyset$ be subsets of $\mathbb{R}$, and $f: A \rightarrow \mathbb{R}$ and $g: B \rightarrow \mathbb{R}$. Let $J$ be an interval.

If for any $x, y \in A$ and any $p, q > 0, p + q = 1$ such that $px + qy \in B$, we have
\begin{align}
\nonumber
p f(x) + q f(y) + g(px + qy) \in J,
\end{align}
then for any random variable $\nu \in A$ such that $\operatorname{\bf E}\nu \in B$, we have
\begin{align}
\nonumber
\operatorname{\bf E} f(\nu) + g( \operatorname{\bf E}\nu ) \in J.
\end{align}
\end{theorem}

\centerline{ \textbf{Part \uppercase\expandafter{\romannumeral2}} }

\begin{theorem}\label{0406081}
Let $D \neq \emptyset$ be a subset of $\mathbb{R}$, and let $f$ be convex at $b_{1}, \ldots, b_{n}$ over $D$. Then for any
\begin{enumerate}[(1)]
\item $a_{i} \in D$ and $p_{i} > 0, \sum_{i=1}^{m} p_{i} = 1$,
\item $q_{j} > 0, \sum_{j=1}^{n} q_{j} = 1$,
\end{enumerate}
such that $\sum_{i=1}^{m} p_{i}a_{i} = \sum_{j=1}^{n} q_{j}b_{j}$, where $a_{i} \not\in ( \min_{j} b_{j}, \max_{j} b_{j} )$, we have
\begin{align}
\nonumber
\sum_{i=1}^{m} p_{i} f(a_{i}) \geq \sum_{j=1}^{n} q_{j} f(b_{j}).
\end{align}
If $f$ is strictly convex at $b_{1}, \ldots, b_{n}$ over $D$, then the equality holds if and only if $b_{j} = \min_{i} a_{i}\ \text{or}\ \max_{i} a_{i}, j = 1, \ldots, n$.
\end{theorem}


\begin{theorem}\label{0608003}
Let $f: [ a_{1}, b_{1} ] \rightarrow \mathbb{R}$ be convex and $[ a_{0}, b_{0} ] \subset ( a_{1}, b_{1} )$.
Then for any $\lambda(x) \geq 0, \int_{ [a_{1}, b_{1}] \setminus (a_{0}, b_{0}) } \lambda(x) \mathrm{d}x = 1$ and any $w(x) \geq 0, \int_{ a_{0} }^{ b_{0} } w(x) \mathrm{d}x = 1$ such that
\begin{align}
\nonumber
\int_{ [a_{1}, b_{1}] \setminus (a_{0}, b_{0}) } \lambda(x) x \mathrm{d}x = \int_{ a_{0} }^{ b_{0} } w(x)x \mathrm{d}x,
\end{align}
we have
\begin{align}
\nonumber
\int_{ [a_{1}, b_{1}] \setminus (a_{0}, b_{0}) } \lambda(x) f(x) \mathrm{d}x \geq \int_{ a_{0} }^{ b_{0} } w(x) f(x) \mathrm{d}x.
\end{align}
If $f$ is strictly convex, then the inequality is strict.
\end{theorem}


\begin{theorem}\label{0701004}
Let $A, B \neq \emptyset$ be subsets of $\mathbb{R}$, and let $f: A \rightarrow \mathbb{R}$ and $g: B \rightarrow \mathbb{R}$.
Let $J$ be an interval.

If for any $x, y \in A$ and any $p, q > 0, p + q = 1$ such that $px + qy \in B$, we have
\begin{align}
\nonumber
p f(x) + q f(y) +g(px + qy) \in J,
\end{align}
then for any
\begin{enumerate}[(1)]
\item $a_{i} \in A$ and $p_{i} > 0, \sum_{i=1}^{m} p_{i} = 1$,
\item $b_{j} \in B$ and $q_{j} > 0, \sum_{j=1}^{n} q_{j} = 1$,
\end{enumerate}
such that $\sum_{i=1}^{m} p_{i}a_{i} = \sum_{j=1}^{n} q_{j}b_{j}$, where $a_{i} \not\in ( \min_{j} b_{j}, \max_{j} b_{j} )$, we have
\begin{align}
\nonumber
\sum_{i=1}^{m} p_{i} f(a_{i}) + \sum_{j=1}^{n} q_{j} g(b_{j}) \in J.
\end{align}
\end{theorem}
\begin{proof}By Theorems \ref{0902001} and \ref{0709002}, we conclude the proof.
\end{proof}


\begin{theorem}\label{0902017}
Let $f \in \mathcal{R}( [ a_{1}, b_{1} ] \setminus ( a_{0}, b_{0} ) )$, and let $g \in \mathcal{R}[ a_{0}, b_{0} ]$, where $[ a_{0}, b_{0} ] \subset ( a_{1}, b_{1} )$.
Let $J$ be an interval.

If for any $x, y \in [a_{1}, b_{1}] \setminus (a_{0}, b_{0})$ and any $p, q > 0, p + q = 1$ such that $px + qy \in [a_{0}, b_{0}]$, we have
\begin{align}
\nonumber
p f(x) + q f(y) + g(px + qy) \in J,
\end{align}
then for any $\lambda(x) \geq 0, \int_{ [a_{1}, b_{1}] \setminus (a_{0}, b_{0}) } \lambda(x) \mathrm{d}x = 1$ and any $w(x) \geq 0, \int_{ a_{0} }^{ b_{0} } w(x) \mathrm{d}x = 1$ such that
\begin{align}
\nonumber
\int_{ [a_{1}, b_{1}] \setminus (a_{0}, b_{0}) } \lambda(x) x \mathrm{d}x = \int_{ a_{0} }^{ b_{0} } w(x)x \mathrm{d}x,
\end{align}
we have
\begin{align}
\nonumber
\int_{ [a_{1}, b_{1}] \setminus (a_{0}, b_{0}) } \lambda(x) f(x) \mathrm{d}x + \int_{ a_{0} }^{ b_{0} } w(x) g(x) \mathrm{d}x \in J.
\end{align}
\end{theorem}
\begin{flushleft}
$1^{0}$. The theorem is a special case of Theorem \ref{0715003}.
\end{flushleft}


\begin{theorem}\label{0711006}
Let $A, B \neq \emptyset$ be subsets of $\mathbb{R}$, and let $f: A \rightarrow \mathbb{R}$ and $g: B \rightarrow \mathbb{R}$.
Let $J$ be an interval.

If for any $x, y \in A$ and any $p, q > 0, p + q = 1$ such that $px + qy \in B$, we have
\begin{align}
\nonumber
p f(x) + q f(y) +g(px + qy) \in J,
\end{align}
then for any random variables $\mu \in A$ and $\nu \in B$ such that
\begin{enumerate}[(1)]
\item $\operatorname{\bf E}\mu = \operatorname{\bf E}\nu$,
\item there are real numbers $m$ and $M$ such that $\mu \not\in (m, M)$ and $\nu \in [m, M]$,
\end{enumerate}
we have
\begin{align}
\nonumber
\operatorname{\bf E} f(\mu) + \operatorname{\bf E} g(\nu) \in J.
\end{align}
\end{theorem}

\section{Conditional convex hull}

\begin{theorem}\label{0407078}
Let $\mathscr{A}$ be a set of subsets of $V$, where $V$ is a real vector space. Then
\begin{align}
\nonumber
\left( \bigcup_{ A \in \mathscr{A} } A^{*} \right)^{*} = \left( \bigcup_{ A \in \mathscr{A} } A \right)^{*}.
\end{align}
\end{theorem}

\centerline{ \textbf{Part \uppercase\expandafter{\romannumeral1}} }

\subsection{The $f^{*}(D, B)$ type}

\begin{definition}\label{0630006}
Let $D \neq \emptyset$ be a subset of $\mathbb{R}$, and let $f: D \rightarrow \mathbb{R}$. Suppose that $B \subseteq D^{*}$. Define
\begin{align}
\nonumber
f^{*}(D, B) = \left\{ \lambda_{1}f(x_{1}) + \cdots + \lambda_{n}f(x_{n})
\left\vert\begin{array}{c}
x_{1}, \ldots, x_{n} \in D, \\
\lambda_{1}, \ldots, \lambda_{n} > 0, \sum_{i=1}^{n}\lambda_{i} = 1, \\
\lambda_{1}x_{1} + \cdots + \lambda_{n}x_{n} \in B, \\
n = 1, 2, \ldots
\end{array}\right. \right\}.
\end{align}
If $B = \{c\}$, then $f^{*}(D, c)$ for short; and if $B = D^{*}$, then $f^{*}(D)$ for short.
\end{definition}
\begin{flushleft}
$1^{0}$. $f^{*}(D) = ( f(D) )^{*}$.
\end{flushleft}


\begin{theorem}\label{0630005}
Let $D \neq \emptyset$ be a subset of $\mathbb{R}$, and let $f: D \rightarrow \mathbb{R}$. Suppose that $B \subseteq D^{*}$.
If $B$ is an interval, then $f^{*}(D, B)$ is also an interval.
\end{theorem}
\begin{flushleft}
$1^{0}$. $f^{*}(D)$ is an interval. \\
$2^{0}$. $f^{*}(D, c), c \in D^{*}$ are intervals.
\end{flushleft}

\subsection{The $f_{\infty}^{*}(D, B)$ type}

\begin{definition}\label{0630002}
Let $f \in \mathcal{R}[a, b]$. Suppose that $B \subseteq (a, b)$. Define
\begin{align}
\nonumber
f_{\infty}^{*}([a, b], B) = \left\{ \int_{a}^{b} \lambda(x) f(x) \mathrm{d}x
\left\vert\begin{array}{c}
\lambda(x) \geq 0, \int_{a}^{b} \lambda(x) \mathrm{d}x = 1, \\
\int_{a}^{b} \lambda(x) x \mathrm{d}x \in B
\end{array}\right. \right\}.
\end{align}
If $B = \{c\}$, then $f_{\infty}^{*}([a, b], c)$ for short; and if $B = (a, b)$, then $f_{\infty}^{*}([a, b])$ for short.
\end{definition}


\begin{theorem}\label{0511001}
Let $f \in \mathcal{R}[a, b]$. Suppose that $B \subseteq (a, b)$.
If $B$ is an interval, then $f_{\infty}^{*}(D, B)$ is also an interval.
\end{theorem}
\begin{flushleft}
$1^{0}$. $f_{\infty}^{*}( [a, b] )$ is an interval. \\
$2^{0}$. $f_{\infty}^{*}( [a, b], c ), c \in (a, b)$ are intervals.
\end{flushleft}

\subsection{The connection between $f^{*}( [a, b], B )$ and $f_{\infty}^{*}( [a, b], B )$}

\begin{theorem}\label{0811010}
Let $f \in \mathcal{R}[a, b]$. Suppose that $c \in (a, b)$.

If for any $x, y \in [a, b]$ and any $p, q > 0, p + q = 1$ such that $px + qy = c$, we have
\begin{align}
\nonumber
p f(x) + q f(y) \geq 0,
\end{align}
then for any $\lambda(x) \geq 0, \int_{a}^{b} \lambda(x) \mathrm{d}x = 1$ such that $\int_{a}^{b} \lambda(x) x \mathrm{d}x = c$, we have
\begin{align}
\nonumber
\int_{a}^{b} \lambda(x) f(x) \mathrm{d}x \geq 0.
\end{align}
\end{theorem}
\begin{proof}For any partition $P = \{ x_{0}, x_{1}, \ldots, x_{n} \}$ of $[a, b]$, there are $c_{i} \in ( x_{i-1}, x_{i} )$ such that
\begin{align}
\nonumber
\int_{ x_{i-1} }^{ x_{i} } \lambda(x) x \mathrm{d}x = c_{i}\int_{ x_{i-1} }^{ x_{i} } \lambda(x) \mathrm{d}x := \lambda_{i}c_{i}, i = 1, \ldots, n.
\end{align}
Then $c_{i} \in [a, b]$ and $\lambda_{i} \geq 0, \sum_{i=1}^{n}\lambda_{i} = 1$ such that $\sum_{i=1}^{n}\lambda_{i}c_{i} = c$.
By Theorem \ref{0709002},
\begin{align}
\nonumber
\sum_{i=1}^{n} f(c_{i})\int_{ x_{i-1} }^{ x_{i} } \lambda(x) \mathrm{d}x = \sum_{i=1}^{n}\lambda_{i} f(c_{i}) \geq 0.
\end{align}
Then
\begin{align}
\nonumber
\int_{a}^{b} \lambda(x) f(x) \mathrm{d}x = \lim_{ \|P\| \rightarrow 0 }\sum_{i=1}^{n} f(c_{i})\int_{ x_{i-1} }^{ x_{i} } \lambda(x) \mathrm{d}x \geq 0.
\end{align}
This concludes the proof.
\end{proof}


\begin{theorem}\label{0811008}
Let $f \in \mathcal{R}[a, b]$. Suppose that $c \in (a, b)$.

If for any $x, y \in [a, b]$ and any $p, q > 0, p + q = 1$ such that $px + qy = c$, we have
\begin{align}
\nonumber
p f(x) + q f(y) > 0,
\end{align}
then for any $\lambda(x) \geq 0, \int_{a}^{b} \lambda(x) \mathrm{d}x = 1$ such that $\int_{a}^{b} \lambda(x) x \mathrm{d}x = c$, we have
\begin{align}
\nonumber
\int_{a}^{b} \lambda(x) f(x) \mathrm{d}x > 0.
\end{align}
\end{theorem}
\begin{proof}Note that $\int_{a}^{c} \lambda(x) \mathrm{d}x, \int_{c}^{b} \lambda(x) \mathrm{d}x > 0$.

By Lebesgue's criterion, we choose two points $x_{0} \in (a, c)$ and $y_{0} \in (c, b)$ such that $\lambda$ and $f$ are continuous at both $x_{0}$ and $y_{0}$, and $\lambda(x_{0}), \lambda(y_{0}) > 0$.

Let $p_{0}, q_{0} >0, p_{0} + q_{0} = 1$ such that $p_{0}x_{0} + q_{0}y_{0} = c$. Then
\begin{align}
\nonumber
p_{0} f(x_{0}) + q_{0} f(y_{0}) > 0.
\end{align}
We choose $x_{0} \in ( a_{1}, b_{1} ) \subset (a, c)$ and $y_{0} \in ( a_{2}, b_{2} ) \subset (c, b)$ such that
\begin{enumerate}[(1)]
\item $f(x) > f(y) - \frac{ p_{0} f(x_{0}) + q_{0} f(y_{0}) }{2}, x, y \in [ a_{1}, b_{1} ]\ \text{or}\ x, y \in [ a_{2}, b_{2} ]$.
\item For any $x \in [a_{1}, b_{1}], y \in [a_{2}, b_{2}]$ and any $p, q > 0, p + q = 1$ such that $px + qy = c$, we have $p f(x) + q f(y) > \frac{ p_{0} f(x_{0}) + q_{0} f(y_{0}) }{2}$.
\end{enumerate}

There are $c_{1} \in ( a_{1}, b_{1} )$ and $c_{2} \in ( a_{2}, b_{2}) $ such that
\begin{align}
\nonumber
\int_{ a_{1} }^{ b_{1} } \lambda(x) x \mathrm{d}x = c_{1}\int_{ a_{1} }^{ b_{1} } \lambda(x) \mathrm{d}x\ \text{and}\ \int_{ a_{2} }^{ b_{2} } \lambda(x) x \mathrm{d}x = c_{2}\int_{ a_{2} }^{ b_{2} } \lambda(x) \mathrm{d}x.
\end{align}
Note that $\int_{ a_{1} }^{ b_{1} } \lambda(x) \mathrm{d}x, \int_{ a_{2} }^{ b_{2} } \lambda(x) \mathrm{d}x > 0$.
Choose $\theta_{1}, \theta_{2} \in (0, 1)$ such that
\begin{align}
\nonumber
\frac{ \theta_{1}\int_{ a_{1} }^{ b_{1} } \lambda(x) \mathrm{d}x \cdot c_{1} + \theta_{2}\int_{ a_{2} }^{ b_{2} } \lambda(x) \mathrm{d}x \cdot c_{2} }{ \theta_{1}\int_{ a_{1} }^{ b_{1} } \lambda(x) \mathrm{d}x + \theta_{2}\int_{ a_{2} }^{ b_{2} } \lambda(x) \mathrm{d}x } = c.
\end{align}
Then
\begin{align}
\nonumber
  &\frac{ \theta_{1}\int_{ a_{1} }^{ b_{1} } \lambda(x) f(x) \mathrm{d}x + \theta_{2}\int_{ a_{2} }^{ b_{2} } \lambda(x) f(x) \mathrm{d}x }{ \theta_{1}\int_{ a_{1} }^{ b_{1} } \lambda(x) \mathrm{d}x + \theta_{2}\int_{ a_{2} }^{ b_{2} } \lambda(x) \mathrm{d}x } \\
\nonumber
> &\frac{ \theta_{1}\int_{ a_{1} }^{ b_{1} } \lambda(x) \mathrm{d}x \cdot f(c_{1}) + \theta_{2}\int_{ a_{2} }^{ b_{2} } \lambda(x) \mathrm{d}x \cdot f(c_{2}) }{ \theta_{1}\int_{ a_{1} }^{ b_{1} } \lambda(x) \mathrm{d}x + \theta_{2}\int_{ a_{2} }^{ b_{2} } \lambda(x) \mathrm{d}x } - \frac{ p_{0} f(x_{0}) + q_{0} f(y_{0}) }{2} > 0.
\end{align}
Let
\begin{align}
\nonumber
\lambda_{0}(x) = \left\{\begin{array}{cc}
( 1 - \theta_{1} )\lambda(x), & a_{1} \leq x \leq b_{1}, \\
( 1 - \theta_{2} )\lambda(x), & a_{2} \leq x \leq b_{2}, \\
\lambda(x), &\ \text{others}.
\end{array}\right.
\end{align}
By Theorem \ref{0811010}, $\int_{a}^{b} \lambda_{0}(x) f(x) \mathrm{d}x \geq 0$ because $\int_{a}^{b} \lambda_{0}(x)x \mathrm{d}x = c\int_{a}^{b} \lambda_{0}(x) \mathrm{d}x$. Then
\begin{align}
\nonumber
\int_{a}^{b} \lambda(x) f(x) \mathrm{d}x \geq \theta_{1}\int_{ a_{1} }^{ b_{1} } \lambda(x) f(x) \mathrm{d}x + \theta_{2}\int_{ a_{2} }^{ b_{2} } \lambda(x) f(x) \mathrm{d}x > 0.
\end{align}
This concludes the proof.
\end{proof}


\begin{theorem}\label{0709001}
Let $f \in \mathcal{R}[a, b]$. Suppose $c \in (a, b)$, and let $J$ be an interval.

If for any $x, y \in [a, b]$ and any $p, q > 0, p + q = 1$ such that $px + qy = c$, we have
\begin{align}
\nonumber
p f(x) + q f(y) \in J,
\end{align}
then for any $\lambda(x) \geq 0, \int_{a}^{b} \lambda(x) \mathrm{d}x = 1$ such that $\int_{a}^{b} \lambda(x) x \mathrm{d}x = c$, we have
\begin{align}
\nonumber
\int_{a}^{b} \lambda(x) f(x) \mathrm{d}x \in J.
\end{align}
\end{theorem}
\begin{proof}By Theorems \ref{0811010} and \ref{0811008}, we conclude the proof.
\end{proof}


\begin{theorem}\label{0529007}
Let $f \in \mathcal{R}[a, b]$. Then
\begin{align}
\nonumber
\emptyset \neq f_{\infty}^{*}( [a, b], c ) \subseteq f^{*}( [a, b], c ), c \in (a, b).
\end{align}
\end{theorem}
\begin{proof}By Theorems \ref{0630005} and \ref{0709001}, we conclude the proof.
\end{proof}
\begin{flushleft}
$1^{0}$. $\emptyset \neq f_{\infty}^{*}([a, b]) \subseteq f^{*}([a, b])$. \\
$2^{0}$. $\emptyset \neq f_{\infty}^{*}( [a, b], B ) \subseteq f^{*}( [a, b], B ), \emptyset \neq B \subseteq (a, b)$.
\end{flushleft}


\begin{flushleft}
Similarly, we have the following theorem.
\end{flushleft}
\begin{theorem}\label{0715001}
Let $f \in \mathcal{R}( [ a_{1}, b_{1} ] \setminus ( a_{0}, b_{0} ) )$, where $[ a_{0}, b_{0} ] \subset ( a_{1}, b_{1} )$. Then
\begin{align}
\nonumber
\emptyset \neq f_{\infty}^{*}( [ a_{1}, b_{1} ] \setminus ( a_{0}, b_{0} ), c ) \subseteq f^{*}( [ a_{1}, b_{1} ] \setminus ( a_{0}, b_{0} ), c ), c \in ( a_{0}, b_{0} ).
\end{align}
\end{theorem}
\begin{flushleft}
$1^{0}$. $\emptyset \neq f_{\infty}^{*}( [ a_{1}, b_{1} ] \setminus ( a_{0}, b_{0} ) ) \subseteq f^{*}( [ a_{1}, b_{1} ] \setminus ( a_{0}, b_{0} ) )$. \\
$2^{0}$. $\emptyset \neq f_{\infty}^{*}( [ a_{1}, b_{1} ] \setminus ( a_{0}, b_{0} ), B ) \subseteq f^{*}( [ a_{1}, b_{1} ] \setminus ( a_{0}, b_{0} ), B ), \emptyset \neq B \subseteq ( a_{0}, b_{0} )$.
\end{flushleft}

\centerline{ \textbf{Part \uppercase\expandafter{\romannumeral2}} }

\subsection{The $f_{n}^{*}(D, B)$ type}

\begin{definition}\label{0117001}
Let $D \neq \emptyset$ be a subset of $\mathbb{R}$, and let $f: D \rightarrow \mathbb{R}$.
Suppose that $B \subseteq D^{*}$. Define
\begin{align}
\nonumber
f_{n}^{*}(D, B) = \left\{ \lambda_{1}f(x_{1}) + \cdots + \lambda_{n}f(x_{n})
\left\vert\begin{array}{c}
x_{1}, \ldots, x_{n} \in D, \\
\lambda_{1}, \ldots, \lambda_{n} > 0, \sum_{i=1}^{n}\lambda_{i} = 1, \\
\lambda_{1}x_{1} + \cdots + \lambda_{n}x_{n} \in B
\end{array}\right. \right\}.
\end{align}
If $B = \{c\}$, then $f_{n}^{*}(D, c)$ for short; and if $B = D^{*}$, then $f_{n}^{*}(D)$ for short.
\end{definition}

\subsection{$f$ is any function}

\begin{theorem}\label{0117003}
Let $D \neq \emptyset$ be a subset of $\mathbb{R}$, and let $f: D \rightarrow \mathbb{R}$. Then
\begin{align}
\nonumber
f_{3}^{*}(D, c) = f^{*}(D, c) = f_{2}^{**}(D, c) \neq \emptyset, c \in D^{*}.
\end{align}
\end{theorem}
\begin{proof}By (\ref{1123001}) we have $\emptyset \neq f_{2}^{*}(D, c) \subseteq f_{3}^{*}(D, c) \subseteq f^{*}(D, c) \subseteq f_{2}^{**}(D, c)$.
Hence, we only need to prove that $f_{3}^{*}(D, c)$ is an interval.
We have
\begin{align}
\nonumber
f_{3}^{*}(D, c) = \left\{ p f(x) + q f(y) + r f(z)
\left\vert\begin{array}{c}
x, y, z \in D, \\
p, q, r \geq 0, p + q + r = 1, \\
px + qy + rz = c
\end{array}\right. \right\}.
\end{align}
Without loss of generality we assume $x, y, z \in D$, $x \leq y \leq z$ and $x \leq c \leq z$. Note that
\begin{align}
\nonumber
&D_{x, y, z} = \left\{ (p, q, r)
\left\vert\begin{array}{c}
p, q, r \geq 0, p + q + r = 1, \\
px + qy + rz = c
\end{array}\right. \right\} \\
\nonumber
= &\{ (p, q, r) \mid p, q, r \geq 0, p + q + r = 1 \} \cap \{ (p, q, r) \mid px + qy + rz = c \} \neq \emptyset
\end{align}
is convex, so connected. Define $h(p, q, r) = p f(x) + q f(y) + r f(z)$. Then
\begin{align}
\nonumber
J_{x, y, z} = h( D_{x, y, z} ) = \left\{ p f(x) + q f(y) + r f(z)
\left\vert\begin{array}{c}
p, q, r \geq 0, p + q + r = 1, \\
px + qy + rz = c
\end{array}\right. \right\}
\end{align}
is an interval.
We now consider the following two cases.
\begin{flushleft}
\textbf{1.} $c \in D$.
\end{flushleft}

Because $J_{x, y, z} \subseteq J_{x, y, c} \bigcup J_{y, z, c} \bigcup J_{z, x, c}$, we have $f_{3}^{*}(D, c) = \bigcup_{x, y \in D} J_{x, y, c}$.
Note that $f(c) \in \bigcap_{x, y \in D} J_{x, y, c}$. By Theorem \ref{0411004}, $f_{3}^{*}(D, c)$ is an interval.
\begin{flushleft}
$\textbf{2.}$ $c \not\in D$.
\end{flushleft}

Let $x, y \in D, x < c < y$. Then $f(x) + \theta( f(y) - f(x) ) \in \bigcap_{z \in D} J_{x, y, z}$, where $x + \theta(y - x) = c$.
Hence, by Theorem \ref{0411004}, $J_{x, y} = \bigcup_{z \in D} J_{x, y, z}$ is an interval.

For any $x_{i}, y_{i} \in D, x_{i} < c < y_{i}, i = 1, 2$, we have
\begin{align}
\nonumber
J_{ x_{1}, y_{1} } \cap J_{ x_{1}, y_{2} } \supseteq J_{ x_{1}, y_{1}, y_{2} } \neq \emptyset\ \text{and}\ J_{ x_{1}, y_{2} } \cap J_{ x_{2}, y_{2} } \supseteq J_{ x_{1}, x_{2}, y_{2} } \neq \emptyset.
\end{align}
Therefore, by Theorem \ref{0411002}, $f_{3}^{*}(D, c) = \bigcup_{ x, y \in D \atop x < c < y } J_{x, y}$ is an interval.
\end{proof}


\begin{example}\label{0405004}
Let $f(x) = 1_{x=0}$ equals $1$ when $x = 0$ and $0$ otherwise.
Then $f_{2}^{*}( \mathbb{R}, 0 ) = \{0, 1\}$ and $f_{3}^{*}( \mathbb{R}, 0 ) = f^{*}( \mathbb{R}, 0 ) = f_{2}^{**}( \mathbb{R}, 0 ) = [0, 1]$.
\end{example}


\begin{theorem}\label{1017002}
Let $D \neq \emptyset$ be a subset of $\mathbb{R}$, and let $f: D \rightarrow \mathbb{R}$. Then
\begin{align}
\nonumber
\emptyset \neq f_{3}^{*}(D, B) = f^{*}(D, B) \subseteq f_{2}^{**}(D, B), \emptyset \neq B \subseteq D^{*},
\end{align}
and equality holds if $B$ is an interval.
\end{theorem}
\begin{proof}By Theorems \ref{0407078} and \ref{0117003}, we have
\begin{align}
\nonumber
f^{*}(D, B) = &\bigcup_{c \in B} f^{*}(D, c) = \bigcup_{c \in B} f_{2}^{**}(D, c) \\
\nonumber
    \subseteq &\left( \bigcup_{c \in B} f_{2}^{**}(D, c) \right)^{*} = \left( \bigcup_{c \in B} f_{2}^{*}(D, c) \right)^{*} = f_{2}^{**}(D, B).
\end{align}
This concludes the proof.
\end{proof}
\begin{flushleft}
$1^{0}$. If we set $B = \{c\}$, then the Theorem \ref{0117003}.
\end{flushleft}


\begin{example}\label{0218011}
Let $f(x) = 1_{x = \pm 1}$. Then
\begin{align}
\nonumber
f_{2}^{*}( [-1, 1], [-1/2, 1/2] ) = &[0, 3/4) \cup \{1\}, \\
\nonumber
f_{3}^{*}( [-1, 1], [-1/2, 1/2] ) = &f^{*}( [-1, 1], [-1/2, 1/2] ) = f_{2}^{**}( [-1, 1], [-1/2, 1/2] ) = [0, 1].
\end{align}
\end{example}

\subsection{$f$ is continuous}

\begin{theorem}\label{0529002}
Let $I$ be an interval, and let $f: I \rightarrow \mathbb{R}$ be continuous. Then
\begin{align}
\nonumber
f_{2}^{*}(I, c) = f^{*}(I, c) \neq \emptyset, c \in I.
\end{align}
\end{theorem}
\begin{proof}Let
\begin{align}
\nonumber
J_{x} = \left\{ p f(x) + q f(y)
\left\vert\begin{array}{c}
y \in I, \\
p, q \geq 0, p + q = 1, \\
px + qy = c
\end{array}\right. \right\}, x \in I.
\end{align}
Note that $f(c) \in \bigcap_{x \in I} J_{x}$, where $J_{x}$ are intervals.
Then by Theorem \ref{0411004}, $f_{2}^{*}(I, c) = \bigcup_{x \in I} J_{x}$ is an interval.
This concludes the proof.
\end{proof}


\begin{theorem}\label{0711002}
Let $I$ be an interval, and let $f: I \rightarrow \mathbb{R}$ be continuous. Then
\begin{align}
\nonumber
\emptyset \neq f_{2}^{*}(I, B) = f^{*}(I, B) \subseteq f_{2}^{**}(I, B), \emptyset \neq B \subseteq I,
\end{align}
and equality holds if $B$ is an interval.
\end{theorem}
\begin{flushleft}
$1^{0}$. If we set $B = \{ c \}$, then the Theorem \ref{0529002}.
\end{flushleft}

\section{Applications of conditional convex hull}

\centerline{ \textbf{Part \uppercase\expandafter{\romannumeral1}} }

\begin{theorem}\label{0209001}
Let $D \neq \emptyset$ be a subset of $\mathbb{R}$, and let $f: D \rightarrow \mathbb{R}$. Suppose that $\emptyset \neq B \subseteq D$. Then
\begin{align}
\nonumber
  &\left\{ \sum_{i=1}^{n}\lambda_{i} f(x_{i}) - f\left( \sum_{i=1}^{n}\lambda_{i}x_{i} \right)
\left\vert\begin{array}{c}
x_{1}, \ldots, x_{n} \in D, \\
\lambda_{1}, \ldots, \lambda_{n} > 0, \sum_{i=1}^{n}\lambda_{i} = 1, \\
\sum_{i=1}^{n}\lambda_{i}x_{i} \in B
\end{array}\right. \right\} \\
\nonumber
= &\left\{ p f(a) + q f(b) - f(pa + qb)
\left\vert\begin{array}{c}
a, b \in D, \\
p, q > 0, p + q = 1, \\
pa + qb \in B
\end{array}\right. \right\}^{*} \ni 0, n \geq 3.
\end{align}
\end{theorem}
\begin{proof}Let $n \in \mathbb{N}$ and $n \geq 3$.
By Theorem \ref{0117003}, $f_{n}^{*}(D, c), c \in B$ are intervals.
Note that $0 \in \bigcap_{c \in B}( f_{n}^{*}(D, c) - f(c) )$.
Then, by Theorem \ref{0411004},
\begin{align}
\nonumber
  &\left\{ \sum_{i=1}^{n}\lambda_{i} f(x_{i}) - f\left( \sum_{i=1}^{n}\lambda_{i}x_{i} \right)
\left\vert\begin{array}{c}
x_{1}, \ldots, x_{n} \in D, \\
\lambda_{1}, \ldots, \lambda_{n} > 0, \sum_{i=1}^{n}\lambda_{i} = 1, \\
\sum_{i=1}^{n}\lambda_{i}x_{i} \in B
\end{array}\right. \right\} \\
\nonumber
= &\bigcup_{c \in B}( f_{n}^{*}(D, c) - f(c) )
\end{align}
is an interval. This concludes the proof.
\end{proof}
\begin{flushleft}
$1^{0}$. By Theorem \ref{0529007},
\begin{align}
\nonumber
  &\left\{ \operatorname{\bf E} f(\nu) - f( \operatorname{\bf E}\nu )
\left\vert\begin{array}{c}
\nu \in D\ \text{is a random variable}, \\
\operatorname{\bf E}\nu \in B
\end{array}\right. \right\} \\
\nonumber
= &\left\{ p f(a) + q f(b) - f(pa + qb)
\left\vert\begin{array}{c}
a, b \in D, \\
p, q > 0, p + q = 1, \\
pa + qb \in B
\end{array}\right. \right\}^{*}.
\end{align}
\end{flushleft}


\begin{theorem}\label{0801001}
Let $I$ be an interval, and let $f: I \rightarrow \mathbb{R}$ be continuous.
Suppose that $\emptyset \neq B \subseteq I$. Then
\begin{align}
\nonumber
  &\left\{ \sum_{i=1}^{n}\lambda_{i} f(x_{i}) - f\left( \sum_{i=1}^{n}\lambda_{i}x_{i} \right)
\left\vert\begin{array}{c}
x_{1}, \ldots, x_{n} \in I, \\
\lambda_{1}, \ldots, \lambda_{n} > 0, \sum_{i=1}^{n}\lambda_{i} = 1, \\
\sum_{i=1}^{n}\lambda_{i}x_{i} \in B
\end{array}\right. \right\} \\
\nonumber
= &\left\{ p f(x) + q f(y) - f(px + qy)
\left\vert\begin{array}{c}
x, y \in I, \\
p, q > 0, p + q = 1, \\
px + qy \in B
\end{array}\right. \right\}^{*} \ni 0, n \geq 2.
\end{align}
\end{theorem}
\begin{proof}By Theorems \ref{0529002} and \ref{0411004}, we conclude the proof.
\end{proof}
\begin{flushleft}
$1^{0}$. By Theorem \ref{0529007},
\begin{align}
\nonumber
  &\left\{ \operatorname{\bf E} f(\nu) - f( \operatorname{\bf E}\nu )
\left\vert\begin{array}{c}
\nu \in I\ \text{is a random variable}, \\
\operatorname{\bf E}\nu \in B
\end{array}\right. \right\} \\
\nonumber
= &\left\{ p f(x) + q f(y) - f(px + qy)
\left\vert\begin{array}{c}
x, y \in I, \\
p, q > 0, p + q = 1, \\
px + qy \in B
\end{array}\right. \right\}.
\end{align}
In particular,
\begin{align}
\nonumber
  &\{ \operatorname{\bf E} f(\nu) - f( \operatorname{\bf E}\nu ) \mid \nu \in I\ \text{is a random variable} \} \\
\nonumber
= &\left\{ p f(x) + q f(y) - f(px + qy)
\left\vert\begin{array}{c}
x, y \in I, \\
p, q > 0, p + q = 1
\end{array}\right. \right\}.
\end{align}
\end{flushleft}


\begin{example}\label{0218012}
Let $f(x) = 1_{x = \pm 1}$ (Clearly, $f$ is convex on $[-1, 1]$). Then
\begin{align}
\nonumber
  &\left\{ p f(x) + q f(y) - f(px + qy)
\left\vert\begin{array}{c}
x, y \in [-1, 1], \\
p, q > 0, p + q = 1, \\
px + qy \in [-1/2, 1/2]
\end{array}\right. \right\} = [0, 3/4) \cup \{1\}, \\
\nonumber
  &\left\{ \sum_{i=1}^{n}\lambda_{i} f(x_{i}) - f\left( \sum_{i=1}^{n}\lambda_{i}x_{i} \right)
\left\vert\begin{array}{c}
x_{1}, \ldots, x_{n} \in [-1, 1], \\
\lambda_{1}, \ldots, \lambda_{n} > 0, \sum_{i=1}^{n}\lambda_{i} = 1, \\
\sum_{i=1}^{n}\lambda_{i}x_{i} \in [-1/2, 1/2]
\end{array}\right. \right\} \\
\nonumber
= &\left( [0, 3/4) \cup \{1\} \right)^{*} = [0, 1], n \geq 3.
\end{align}
\end{example}


\begin{example}\label{0221001}
Let $f(x) = 1_{x=0}$ equals $1$ when $x = 0$ and $0$ otherwise. Then
\begin{align}
\nonumber
  &\left\{ p f(x) + q f(y) - f(px + qy)
\left\vert\begin{array}{c}
x, y \in \mathbb{R}, \\
p, q > 0, p + q = 1
\end{array}\right. \right\} = \{-1\} \cup [0, 1), \\
\nonumber
  &\left\{ \left. \sum_{i=1}^{n}\lambda_{i} f(x_{i}) - f\left( \sum_{i=1}^{n}\lambda_{i}x_{i} \right) \right\vert
\begin{array}{c}
x_{1}, \ldots, x_{n} \in \mathbb{R}, \\
\lambda_{1}, \ldots, \lambda_{n} > 0, \sum_{i=1}^{n}\lambda_{i} = 1
\end{array} \right\} \\
\nonumber
= &( \{-1\} \cup [0, 1) )^{*} = [-1, 1), n \geq 3.
\end{align}
\end{example}

\centerline{ \textbf{Part \uppercase\expandafter{\romannumeral2}} }

\begin{theorem}\label{1019009}
Let $D \neq \emptyset$ be a subset of $\mathbb{R}$, and let $f: D \rightarrow \mathbb{R}$.
Let $g: D^{*} \rightarrow \mathbb{R}$ be continuous. Then
\begin{align}
\nonumber
  &\left\{ \left. \sum_{i=1}^{n}\lambda_{i} f(x_{i}) + g\left( \sum_{i=1}^{n}\lambda_{i}x_{i} \right) \right\vert
\begin{array}{c}
x_{1}, \ldots, x_{n} \in D, \\
\lambda_{1}, \ldots, \lambda_{n} > 0, \sum_{i=1}^{n}\lambda_{i} = 1
\end{array} \right\} \\
\nonumber
= &\left\{ p f(a) + q f(b) + g(pa + qb)
\left\vert\begin{array}{c}
a, b \in D, \\
p, q > 0, p + q = 1
\end{array}\right. \right\}^{*} \neq \emptyset, n \geq 2.
\end{align}
\end{theorem}
\begin{proof}Let
\begin{align}
\nonumber
J_{x, y} = \left\{ p f(x) + q f(y) + g(px + qy) \mid p, q \geq 0, p + q = 1 \right\}, x, y \in D.
\end{align}
Note that $J_{x, y}, x, y \in D$ are intervals, and $f(x) + g(x) \in \bigcap _{y \in D} J_{x, y}, x \in D$.
Then, by Theorem \ref{0411004}, $J_{x} = \bigcup_{y \in D} J_{x, y}, x \in D$ are intervals.
Note that $J_{x} \cap J_{y} \supseteq J_{x, y} \neq \emptyset, x, y \in D$. Then, by Theorem \ref{0411003},
\begin{align}
\nonumber
\left\{ p f(a) + q f(b) + g(pa + qb)
\left\vert\begin{array}{c}
a, b \in D, \\
p, q > 0, p + q = 1
\end{array}\right. \right\} = \bigcup_{x \in D} J_{x}
\end{align}
is an interval. This concludes the proof.
\end{proof}


\begin{theorem}\label{0215002}
Let $D \neq \emptyset$ be a subset of $\mathbb{R}$, and let $f: D \rightarrow \mathbb{R}$.
Let $B$ be a subinterval of $D^{*}$, and let $g: B \rightarrow \mathbb{R}$ be continuous. Then
\begin{align}
\nonumber
 &\left\{ \sum_{i=1}^{n}\lambda_{i} f(x_{i}) + g\left( \sum_{i=1}^{n}\lambda_{i}x_{i} \right)
\left\vert\begin{array}{c}
x_{1}, \ldots, x_{n} \in D, \\
\lambda_{1}, \ldots, \lambda_{n} > 0, \sum_{i=1}^{n}\lambda_{i} = 1, \\
\sum_{i=1}^{n}\lambda_{i}x_{i} \in B
\end{array}\right. \right\} \\
\nonumber
= &\left\{ p f(a) + q f(b) + g(pa + qb)
\left\vert\begin{array}{c}
a, b \in D, \\
p, q > 0, p + q = 1, \\
pa + qb \in B
\end{array}\right. \right\}^{*} \neq \emptyset, n \geq 3.
\end{align}
\end{theorem}
\begin{proof}Note that $\varphi(x, y) = x + g(y)$ is continuous on
\begin{align}
\nonumber
\mathcal{J} = \left\{ \left( \sum_{i=1}^{k}\lambda_{i} f(x_{i}), \sum_{i=1}^{k}\lambda_{i}x_{i} \right)
\left\vert\begin{array}{c}
x_{1}, \ldots, x_{k} \in D, \\
\lambda_{1}, \ldots, \lambda_{k} > 0, \sum_{i=1}^{k}\lambda_{i} = 1, \\
\sum_{i=1}^{k}\lambda_{i}x_{i} \in B, \\
k = 1, 2, \ldots
\end{array}\right. \right\},
\end{align}
where $\mathcal{J}$ is convex, and then connected. Then $\varphi( \mathcal{J} )$ is an interval.

By Theorem \ref{0117003}, for any $n \in \mathbb{N}$ such that $n \geq 3$, we have
\begin{align}
\nonumber
 &\left\{ \sum_{i=1}^{n}\lambda_{i} f(x_{i}) + g\left( \sum_{i=1}^{n}\lambda_{i}x_{i} \right)
\left\vert\begin{array}{c}
x_{1}, \ldots, x_{n} \in D, \\
\lambda_{1}, \ldots, \lambda_{n} > 0, \sum_{i=1}^{n}\lambda_{i} = 1, \\
\sum_{i=1}^{n}\lambda_{i}x_{i} \in B
\end{array}\right. \right\} \\
\nonumber
= &\bigcup_{c \in B}( f_{n}^{*}(D, c) + g(c) ) = \bigcup_{c \in B}( f^{*}(D, c) + g(c) ) \\
\nonumber
= &\left\{ \sum_{i=1}^{k}\lambda_{i} f(x_{i}) + g\left( \sum_{i=1}^{k}\lambda_{i}x_{i} \right)
\left\vert\begin{array}{c}
x_{1}, \ldots, x_{k} \in D, \\
\lambda_{1}, \ldots, \lambda_{k} > 0, \sum_{i=1}^{k}\lambda_{i} = 1, \\
\sum_{i=1}^{k}\lambda_{i}x_{i} \in B, \\
k = 1, 2, \ldots
\end{array}\right. \right\} = \varphi( \mathcal{J} ).
\end{align}
Then, by (\ref{0103001}), we conclude the proof.
\end{proof}

\centerline{ \textbf{Part \uppercase\expandafter{\romannumeral3}} }

\begin{theorem}\label{1019008}
Let $f: D \subseteq \mathbb{R} \rightarrow \mathbb{R}$, and let $g: B \subseteq D^{*} \rightarrow \mathbb{R}$. Then
\begin{align}
\nonumber
\emptyset \neq &\left\{ \sum_{i=1}^{n}\lambda_{i} f(x_{i}) + g\left( \sum_{i=1}^{n}\lambda_{i}x_{i} \right)
\left\vert\begin{array}{c}
x_{1}, \ldots, x_{n} \in D, \\
\lambda_{1}, \ldots, \lambda_{n} > 0, \sum_{i=1}^{n}\lambda_{i} = 1, \\
\sum_{i=1}^{n}\lambda_{i}x_{i} \in B, \\
n = 1, 2, \ldots
\end{array}\right. \right\} \\
\nonumber
\subseteq &\left\{ p f(a) + q f(b) + g(pa + qb)
\left\vert\begin{array}{c}
a, b \in D, \\
p, q > 0, p + q = 1, \\
pa + qb \in B
\end{array}\right. \right\}^{*},
\end{align}
and equality holds if $B$ is an interval and $g$ is continuous.
\end{theorem}
\begin{proof}By Theorems \ref{0407078} and \ref{0117003},
\begin{align}
\nonumber
  &\left\{ \sum_{i=1}^{n}\lambda_{i} f(x_{i}) + g\left( \sum_{i=1}^{n}\lambda_{i}x_{i} \right)
\left\vert\begin{array}{c}
x_{1}, \ldots, x_{n} \in D, \\
\lambda_{1}, \ldots, \lambda_{n} > 0, \sum_{i=1}^{n}\lambda_{i} = 1, \\
\sum_{i=1}^{n}\lambda_{i}x_{i} \in B, \\
n = 1, 2, \ldots
\end{array}\right.
\right\} \\
\nonumber
= &\bigcup_{c \in B}( f^{*}(D, c) + g(c) ) = \bigcup_{c \in B}( f_{2}^{**}(D, c) + g(c) ) = \bigcup_{c \in B}( f_{2}^{*}(D, c) + g(c) )^{*} \\
\nonumber
\subseteq &\left( \bigcup_{c \in B}( f_{2}^{*}(D, c) + g(c) )^{*} \right)^{*} = \left( \bigcup_{c \in B}( f_{2}^{*}(D, c) + g(c) ) \right)^{*} \\
\nonumber
= &\left\{ p f(a) + q f(b) + g(pa + qb)
\left\vert\begin{array}{c}
a, b \in D, \\
p, q > 0, p + q = 1, \\
pa + qb \in B
\end{array}\right.
\right\}^{*}.
\end{align}
This concludes the proof.
\end{proof}


\begin{theorem}\label{0606002}
Let $f \in \mathcal{R}[a, b]$, and let $g: B \subseteq (a, b) \rightarrow \mathbb{R}$. Then
\begin{align}
\nonumber
\emptyset \neq &\left\{ \int_{a}^{b} \lambda(x) f(x) \mathrm{d}x + g\left( \int_{a}^{b} \lambda(x) x \mathrm{d}x \right)
\left\vert\begin{array}{c}
\lambda(x) \geq 0, \int_{a}^{b} \lambda(x) \mathrm{d}x = 1, \\
\int_{a}^{b} \lambda(x) x \mathrm{d}x \in B
\end{array}\right. \right\} \\
\nonumber
\subseteq      &\left\{ p f(a) + q f(b) + g(pa + qb)
\left\vert\begin{array}{c}
x, y \in [a, b], \\
p, q > 0, p + q = 1, \\
px + qy \in B
\end{array}\right. \right\}^{*}.
\end{align}
\end{theorem}
\begin{proof}By Theorems \ref{0529007} and \ref{1019008}, we conclude the proof.
\end{proof}


\begin{theorem}\label{0715003}
Let $f \in \mathcal{R}( [ a_{1}, b_{1} ] \setminus ( a_{0}, b_{0} ) )$, and let $g \in \mathcal{R}[ a_{0}, b_{0} ]$, where $[ a_{0}, b_{0} ] \subset ( a_{1}, b_{1} )$.
Suppose that $\emptyset \neq B \subseteq ( a_{0}, b_{0} )$. Then
\begin{footnotesize}
\begin{align}
\nonumber
\emptyset \neq &\left\{ \int_{ [ a_{1}, b_{1} ] \setminus ( a_{0}, b_{0} ) } \lambda(x) f(x) \mathrm{d}x + \int_{ a_{0} }^{ b_{0} } w(x) g(x) \mathrm{d}x
\left\vert\begin{array}{c}
\lambda(x) \geq 0, \int_{ [ a_{1}, b_{1} ] \setminus ( a_{0}, b_{0} ) } \lambda(x) \mathrm{d}x = 1, \\
w(x) \geq 0, \int_{ a_{0} }^{ b_{0} } w(x) \mathrm{d}x = 1, \\
\int_{ [ a_{1}, b_{1} ] \setminus ( a_{0}, b_{0} ) } \lambda(x)x \mathrm{d}x = \int_{ a_{0} }^{ b_{0} } w(x)x \mathrm{d}x \in B
\end{array}\right. \right\} \\
\nonumber
\subseteq      &\left\{ p f(x) + q f(y) + g(px + qy)
\left\vert\begin{array}{c}
x, y \in [ a_{1}, b_{1} ] \setminus ( a_{0}, b_{0} ), \\
p, q > 0, p + q = 1, \\
px + qy \in B
\end{array}\right. \right\}^{*}.
\end{align}
\end{footnotesize}
\end{theorem}
\begin{proof}By Theorems \ref{0701004}, \ref{0529007} and \ref{0715001},
\begin{footnotesize}
\begin{align}
\nonumber
  &\left\{ \int_{ [ a_{1}, b_{1} ] \setminus ( a_{0}, b_{0} ) } \lambda(x) f(x) \mathrm{d}x + \int_{ a_{0} }^{ b_{0} } w(x) g(x) \mathrm{d}x
\left\vert\begin{array}{c}
\lambda(x) \geq 0, \int_{ [ a_{1}, b_{1} ] \setminus ( a_{0}, b_{0} ) } \lambda(x) \mathrm{d}x = 1, \\
w(x) \geq 0, \int_{ a_{0} }^{ b_{0} } w(x) \mathrm{d}x = 1, \\
\int_{ [ a_{1}, b_{1} ] \setminus ( a_{0}, b_{0} ) } \lambda(x)x \mathrm{d}x = \int_{ a_{0} }^{ b_{0} } w(x)x \mathrm{d}x \in B
\end{array}\right. \right\} \\
\nonumber
= &\bigcup_{c \in B}( f_{\infty}^{*}( [ a_{1}, b_{1} ] \setminus ( a_{0}, b_{0} ), c ) + g_{\infty}^{*}( [ a_{0}, b_{0} ], c ) ) \subseteq \bigcup_{c \in B}( f^{*}( [ a_{1}, b_{1} ] \setminus ( a_{0}, b_{0} ), c ) + g^{*}( [ a_{0}, b_{0} ], c ) ) \\
\nonumber
\subseteq &\left\{ p f(x) + q f(y) + g(px + qy)
\left\vert\begin{array}{c}
x, y \in [ a_{1}, b_{1} ] \setminus ( a_{0}, b_{0} ), \\
p, q > 0, p + q = 1, \\
px + qy \in B
\end{array}\right. \right\}^{*}.
\end{align}
\end{footnotesize} \\
This concludes the proof.
\end{proof}

\section{Appendix A}

{\flushleft \textbf{Union of intervals is an interval tests}}

\centerline{ \textbf{Part \uppercase\expandafter{\romannumeral1}} }

\begin{theorem}\label{0411004}
Let $\mathscr{A}$ be a cluster of intervals.

If $\bigcap_{ A \in \mathscr{A} } A \neq \emptyset$, then $\bigcup_{ A \in \mathscr{A} } A$ is an interval.
\end{theorem}


\begin{theorem}\label{0411003}
Let $\mathscr{A}$ be a cluster of intervals.

If $A \cap B \neq \emptyset, A, B \in \mathscr{A}$, then $\bigcup_{ A \in \mathscr{A} } A$ is an interval.
\end{theorem}


\begin{theorem}\label{0723001}
Let $\mathscr{A}$ be a cluster of intervals, and let $J \subseteq \bigcup_{ A \in \mathscr{A} } A$ be an interval.
If $A \cap J \neq \emptyset, A \in \mathscr{A}$, then $\bigcup_{ A \in \mathscr{A} } A$ is an interval.
\end{theorem}


\begin{theorem}\label{0912004}
Let $A_{x}, x \in I$ be a cluster of intervals, and let $f: I \rightarrow \mathbb{R}$ be continuous, where $I$ is an interval.

If $f(x) \in A_{x}, x \in I$, then $\bigcup_{x \in I} A_{x}$ is an interval.
\end{theorem}
\begin{proof}Note that $f(x) \in A_{x} \cap f(I), x \in I$, where $f(I) \subseteq \bigcup_{x \in I} A_{x}$ is an interval.
Then, by Theorem \ref{0723001}, we conclude the proof.
\end{proof}


\begin{theorem}\label{0912001}
Let $A, I$ be intervals, and let $f: I \rightarrow \mathbb{R}$ be continuous.
Then $\bigcup_{x \in I}( A + f(x) )$ is an interval.
\end{theorem}
\begin{proof}Let $x_{0} \in A$ and $\varphi(x) = x_{0} + f(x)$.
Then $\varphi$ is continuous on $I$ and $\varphi(x) \in A + f(x), x \in I$.
By Theorem \ref{0912004}, we conclude the proof.
\end{proof}
\begin{flushleft}
$1^{0}$. If we set $A = \{ 0 \}$, then the {\it intermediate value theorem for continuous functions}.
\end{flushleft}


\begin{theorem}\label{1212001}
Let $A_{x}, B_{x}, x \in D$ be two clusters of intervals such that $A_{x} \subseteq B_{x}, x \in D$.

If $\bigcup_{x \in D} A_{x}$ is an interval, then $\bigcup_{x \in D} B_{x}$ is also an interval.
\end{theorem}
\begin{flushleft}
$1^{0}$. The theorem is a special case of Theorem \ref{0529001}.
\end{flushleft}


\begin{theorem}\label{0529001}
Let $A_{x}, B_{x}, x \in D$ be two clusters of intervals such that
\begin{enumerate}[(1)]
\item $A_{x} \cap B_{x} \neq \emptyset$ for all $x \in D$,
\item $\bigcup_{x \in D} A_{x} \subseteq \bigcup_{x \in D} B_{x}$.
\end{enumerate}
If $\bigcup_{x \in D} A_{x}$ is an interval, then $\bigcup_{x \in D} B_{x}$ is also an interval.
\end{theorem}
\begin{proof}Note that $B_{x} \cap \left( \bigcup_{x \in D} A_{x} \right) \neq \emptyset$ for all $x \in D$.
Then by Theorem \ref{0723001}, we conclude the proof.
\end{proof}

\centerline{ \textbf{Part \uppercase\expandafter{\romannumeral2}} }

\begin{definition}\label{0411001}
Let $\mathscr{A}$ be a cluster of intervals, then $\mathscr{A}$ is said to be connected if for any two intervals $A, B \in \mathscr{A}$, there are corresponding intervals
\begin{align}
\nonumber
A = A_{0}, A_{1}, \ldots, A_{n} = B \in \mathscr{A}\ \text{(} n = n(A, B) \text{)}
\end{align}
 such that
\begin{align}
\nonumber
A_{i-1} \cap A_{i} \neq \emptyset, i = 1, \ldots, n.
\end{align}
\end{definition}


\begin{theorem}\label{0411002}
Let $\mathscr{A}$ be a cluster of intervals.

If $\mathscr{A}$ is connected, then $\bigcup_{ A \in \mathscr{A} } A$ is an interval.
\end{theorem}


\begin{theorem}[Heine-Borel]\label{0411005}
Let $\mathscr{A}$ be a cluster of open intervals.
Then $\bigcup_{ A \in \mathscr{A} } A$ is an open interval if and only if $\mathscr{A}$ is connected.
\end{theorem}
\begin{flushleft}
$1^{0}$. The theorem is equivalent to the {\it Heine-Borel theorem}.
\end{flushleft}

\section{Appendix B}

{\flushleft \textbf{The fundamental theorem of systems of point masses}}

\begin{definition}\label{1213001}
Let $D$ be a subset of $V$, where $V$ is a real vector space. Define
\begin{align}
\nonumber
D_{n}^{*} = \left\{ \lambda_{1}\bm{x}_{1} + \cdots + \lambda_{n}\bm{x}_{n}
\left\vert\begin{array}{c}
\bm{x}_{1}, \ldots, \bm{x}_{n} \in D, \\
\lambda_{1}, \ldots, \lambda_{n} > 0, \sum_{i=1}^{n}\lambda_{i} = 1
\end{array}\right. \right\}.
\end{align}
\end{definition}
\begin{flushleft}
$1^{0}$. In particular, $D_{1}^{*} = D$. \\
$2^{0}$. $D_{n}^{*} \subseteq D_{n+1}^{*}$ for all $n \in \mathbb{N}_{+}$, and $D^{*} = \lim_{n \rightarrow +\infty} D_{n}^{*} = \bigcup_{n=1}^{+\infty} D_{n}^{*}$.
\end{flushleft}


\begin{theorem}[The Carath$\acute{ \textbf{e} }$odory's theorem]\label{0407035}\hfill \\
Let $D \subseteq \mathbb{R}^{n}$. Then $D_{n+1}^{*} = D^{*}$.
\end{theorem}


\begin{definition}\label{0907001}
Let $\bm{x} \in \mathbb{R}^{n}$ be a vector, whose $i$-th element is $x_{i}$.

The $L_{0}$ norm of $\bm{x}$, denoted as $\| \bm{x} \|_{0}$, is the number of non-zero elements in $\bm{x}$, that is,
\begin{align}
\nonumber
\| \bm{x} \|_{0} = \sum_{i=1}^{n} 1_{ x_{i} \neq 0 }.
\end{align}
Strictly speaking, the $L_{0}$ norm is not a norm because it does not satisfy homogeneity.

The $L_{1}$ norm of $\bm{x}$, denoted as $\| \bm{x} \|_{1}$, is defined as $\| \bm{x} \|_{1} = \sum_{i=1}^{n}\vert x_{i} \vert$.
\end{definition}


\begin{theorem}\label{0907002}
Suppose that $\bm{a}_{1}, \ldots, \bm{a}_{k} \in \mathbb{R}^{n}$ ($k > n$).
Let $\bm{w} \in (0, +\infty)^{k}$ be a column vector. Then there are column vectors
\begin{align}
\nonumber
\bm{w}_{i} \in [0, +\infty)^{k}, 1 \leq \| \bm{w}_{i} \|_{0} \leq n+1
\end{align}
such that
\begin{align}
\nonumber
\left\{\begin{array}{c}
\bm{w} = \bm{w}_{1} + \cdots + \bm{w}_{m}, \\
\frac{ \bm{w}_{i}^{T} }{ \| \bm{w}_{i} \|_{1} }\left[\begin{array}{c}
\bm{a}_{1} \\
\vdots \\
\bm{a}_{k}
\end{array}\right] = \frac{ \bm{w}^{T} }{ \| \bm{w} \|_{1} }\left[\begin{array}{c}
\bm{a}_{1} \\
\vdots \\
\bm{a}_{k}
\end{array}\right], i = 1, \ldots, m,
\end{array}\right.
\end{align}
where $\frac{k}{n+1} \leq m \leq k-n$.
\end{theorem}
\begin{proof}
By Theorem \ref{0407035}, without loss of generality we assume
\begin{align}
\nonumber
\left\{\begin{array}{c}
\sum_{s=1}^{ k_{1} }\lambda_{s}\bm{a}_{s} = \frac{ w_{1}\bm{a}_{1} + \cdots + w_{k}\bm{a}_{k} }{ w_{1} + \cdots + w_{k} }, \\
\lambda_{s} > 0, \sum_{s=1}^{ k_{1} }\lambda_{s} = 1,
\end{array}\right.
\end{align}
where $1 \leq k_{1} \leq n + 1$ and $w_{s}$ is the $s$-th element of $\bm{w}$, $s = 1, \ldots, k$. \\
Let $\min\left\{ \frac{ w_{1} }{ \lambda_{1} }, \ldots, \frac{ w_{ k_{1} } }{ \lambda_{ k_{1} } } \right\} = \frac{ w_{ k_{0} } }{ \lambda_{ k_{0} } }$. Then
\begin{align}
\nonumber
w_{s} - \frac{ w_{ k_{0} } }{ \lambda_{ k_{0} } }\lambda_{s} \geq 0, s = 1, \ldots, k_{1}.
\end{align}
Therefore,
\begin{align}
\nonumber
  &w_{1}\bm{a}_{1} + \cdots + w_{k}\bm{a}_{k} \\
\nonumber
= &\sum_{s=1}^{ k_{1} }\frac{ w_{ k_{0} } }{ \lambda_{ k_{0} } }\lambda_{s}\bm{a}_{s} +
\sum_{s=1}^{ k_{0}-1 }\left( w_{s} - \frac{ w_{ k_{0} } }{ \lambda_{ k_{0} } }\lambda_{s} \right)\bm{a}_{s} +
\sum_{ s = k_{0}+1 }^{ k_{1} }\left( w_{s} - \frac{ w_{ k_{0} } }{ \lambda_{ k_{0} } }\lambda_{s} \right)\bm{a}_{s} + \sum_{ s = k_{1}+1 }^{k} w_{s}\bm{a}_{s}.
\end{align}
This concludes the proof.
\end{proof}


\begin{theorem}[The fundamental theorem of systems of point masses]\label{0907003}\hfill \\
Every system of $k > n$ point masses in $\mathbb{R}^{n}$ can be decomposed into at most $k-n$ systems of $n+1$ point masses, and all these systems of $n+1$ point masses have the same center of mass as the original one.
\end{theorem}
\begin{flushleft}
$1^{0}$. We also call the theorem the {\it system of point masses decomposition theorem}.
\end{flushleft}


\begin{theorem}\label{0908009}
Suppose that $\bm{a}_{1}, \ldots, \bm{a}_{k} \in \mathbb{R}^{n}$ ($k > n$).
Let $\bm{w} \in (0, +\infty)^{k}$ be a column vector such that $\| \bm{w} \|_{1} = 1$.
Then there are column vectors
\begin{align}
\nonumber
\bm{w}_{i} \in [0, +\infty)^{k}, 1 \leq \| \bm{w}_{i} \|_{0} \leq n+1, \| \bm{w}_{i} \|_{1} = 1,
\end{align}
integer $\frac{k}{n+1} \leq m \leq k-n$, and numbers $\varpi_{1}, \ldots, \varpi_{m} > 0, \varpi_{1} + \cdots + \varpi_{m} = 1$ such that
\begin{align}
\nonumber
\left\{\begin{array}{c}
\bm{w} = \varpi_{1}\bm{w}_{1} + \cdots + \varpi_{m}\bm{w}_{m}, \\
\bm{w}_{i}^{T}\left[\begin{array}{c}
\bm{a}_{1} \\
\vdots \\
\bm{a}_{k}
\end{array}\right] = \bm{w}^{T}\left[\begin{array}{c}
\bm{a}_{1} \\
\vdots \\
\bm{a}_{k}
\end{array}\right], i = 1, \ldots, m.
\end{array}\right.
\end{align}
\end{theorem}





\end{document}